\documentclass[12pts]{article}

\RequirePackage[OT1]{fontenc}
\RequirePackage[amsthm,amsmath,natbib,noinfoline]{imsart}

\usepackage{latexsym}
\usepackage{amsfonts}
\usepackage{amssymb}
\usepackage{fontenc}
\usepackage{textcomp}
\usepackage{graphics}
\usepackage{color}
\usepackage{colortbl}
\usepackage{graphicx}

\startlocaldefs

    \theoremstyle{plain}
    \numberwithin{equation}{section}
    \newtheorem{thrm}{Theorem}[section]             
    \newtheorem{prop}[thrm]{Proposition}
    
    \newtheorem{cllry}[thrm]{Corollary}
    \newtheorem{lmma}[thrm]{Lemma}
    \newtheorem{remk}[thrm]{Remark}
    
    \newtheorem{exmpl}[thrm]{Example}

    \def\calB{{\mathcal B}}

    \def\calI{{\mathcal I}}

    \def\calR{{\mathcal R}}

    \def\bbr{{\mathbb R}}
    \def\bbe{{\,\mathbb E\,}}
    \def\bbp{{\mathbb P}}

    \def\ed{{\,\stackrel{\frak {D}}{=}\,}}

    \definecolor{Red}{rgb}{0.00, 0.00, 0.00}
    \newcommand{\Red}{\color{Red}}
    \definecolor{DRed}{rgb}{0.0, 0.00, 0.00}
    \newcommand{\DRed}{\color{DRed}}
    \definecolor{Blue}{rgb}{0.00, 0.00, 1.00}
    
    \definecolor{PaleGrey}{rgb}{.6, .6, .6}

\endlocaldefs

\begin{document}
\begin{frontmatter}

\title{Small-time expansions for the transition
distributions of L\'evy processes}
\runtitle{Small-time expansions for L\'evy processes}

\begin{aug}
\author{\fnms{Jos\'e} E. \snm{Figueroa-L\'opez}
}
\and
\author{\fnms{Christian} \snm{Houdr\'e}\thanksref{t2}
}


\thankstext{t2}{
It is  a pleasure to thank Philippe Marchal for helpful comments and for suggesting  us an interesting counterexample.}

\runauthor{J.E. Figueroa-L\'opez and C. Houdr\'e}


\address{Department of Statistics\\
Purdue University\\
W. Lafayette, IN 47906, USA\\
\printead[figueroa@stat.purdue.edu]{e1}
}

\address{Department of Mathematics\\
Georgia Institute of Technology\\
Atlanta, GA, 30332, USA\\
\printead[houdre@math.gatech.edu]{e2}
}
\end{aug}

\begin{abstract}
    Let $X = (X_t)_{t\ge 0}$ be a 
L\'evy process with absolutely continuous L\'evy measure $\nu$.   Small time polynomial expansions of order $n$ in $t$  are obtained for the tails 
$\bbp\left( X_{t}\geq{}y\right)$ of the process, assuming 
smoothness conditions on the L\'evy density away from the origin. 
By imposing additional regularity conditions on the transition density $p_{t}$ of $X_{t}$, 
an explicit expression for the remainder of the approximation is also given. As a byproduct, polynomial expansions of order $n$ in $t$ are derived for the transition densities of the process. 
The conditions imposed on $p_{t}$
require that its derivatives 
remain uniformly bounded away from the origin, as $t\rightarrow{}0$; 
such conditions are shown to be satisfied for symmetric stable L\'evy processes
as well as for other related L\'evy processes of relevance in mathematical finance. 
The expansions seem to correct asymptotics previously reported in the literature.
\end{abstract}





\end{frontmatter}

\section{Introduction}
L\'evy processes are important building blocks in stochastic models whose evolution in time might 
exhibit sudden changes in value. Such models can be constructed in rather general ways, such as 
stochastic differential equations driven by L\'evy processes or time-changes 
of L\'evy processes.  Many of these models have been suggested and heavily studied in the 
area of mathematical finance 
(see \cite{Cont:2003} for an introduction to some of these applications). 

{A L\'evy process} $X=(X_t)_{t\ge 0}$ is typically described in terms of 
a triplet $(\sigma^{2},b,\nu)$ such that the process can be understood as the superposition of a 
Brownian motion with drift, say $\sigma W_{t}+bt$, and a pure-jump component, whose discontinuities 
are determined by $\nu$ in that, the average intensity (per unit time) of jumps whose size fall in a 
given set of values $A$ is $\nu(A)$.  
Thus, for instance, if $\nu((-\infty,0])=0$, then $X$ will exhibit only positive jumps.  
A common assumption in many applications is that
$\nu$ is determined by a function
$s:\bbr\backslash\{0\}\rightarrow [0,\infty)$, called
the \emph{L\'evy density}, in the sense that
\[
    \nu(A):=\int_{A} s(x) dx, \;\; \forall
    A\in\calB(\bbr\backslash\{0\}).
\]
Intuitively, the value of $s$ at $x_{0}$ provides
information on the frequency of jumps with sizes
``close'' to $x_{0}$.

Still, L\'evy models have some important shortcomings for certain applications.  For instance, given that 
typically the law of $X_{t}$ is specified via its characteristic function 
\[
	\varphi_{t}(u):=\bbe e^{iu X_{t}},
\]
neither its density function $p_{t}$ nor its distribution 
function $\bbp\left(X_{t}\leq{}y\right)$ are explicitly given in many cases.  
Therefore, {the computation of} such quantities necessitates numerical or analytical 
approximation methods. In this paper we study, short time, analytical approximations for the tail distributions 
 $\bbp\left(X_{t}\geq y\right)$. 
This type of asymptotic results plays an important role in
the non-parametric estimation of the L\'evy measure
based on high-frequency sampling observations
of the process  as carefully reported in \cite{Fig:2007a} {(see also \cite{Rubin:1959},
\cite{Figueroa:2004}, and
\cite{Woerner:2003}). In Section \ref{Motivation}, we present some of the ideas behind this important application of our results.}

It is a well-known fact that the first order approximation 
is given by $t\nu([y,\infty))$, in the sense that 
\begin{equation}\label{FrstLimit}
    \lim_{t\rightarrow{}0}\frac{1}{t}\;
    \bbp \left(X_{t}\geq y\right)
    =\nu([y,\infty)),
\end{equation}
{provided that $y$ is a point of continuity of $\nu$}
(see, e.g., Chapter 1 of Bertoin \cite{Bertoin}).  
A natural question is then to determine 
the rate of convergence in (\ref{FrstLimit}).  
In case of a compound Poisson process, this rate is $O(t)$, and it is then natural 
to ask whether or not the limit below exists {for general L\'evy processes:}
\begin{equation}\label{TargetLimit}
\lim_{t\rightarrow{}0}
    \frac{1}{t}\left\{\frac{1}{t}\bbp\left( X_{t}\geq{}y\right)
    -\nu([y,\infty))\right\}.
\end{equation}
In this paper, we study the validity of the more general {polynomial} expansion: 
    \begin{equation}\label{GnrlExpansion0}
        \bbp\left(X_{t}\geq{}y\right)=
        \sum_{k=1}^{n} d_{k}\,\frac{t^{k}}{k!}+
        \frac{t^{n+1}}{n!}\mathcal{R}_{n}(t),
    \end{equation}
for certain constants $d_{k}$ and a remainder term $\mathcal{R}_{n}(t)$ 
bounded for $t$ small enough. 
{
Note that in terms of the coefficients of (\ref{GnrlExpansion0}), the limit (\ref{TargetLimit}) 
is given by 
\begin{equation}\label{TargetLimit1}
\frac{d_{2}}{2}=\lim_{t\rightarrow{}0}
    \frac{1}{t}\left\{\frac{1}{t}\bbp\left( X_{t}\geq{}y\right)
    -\nu([y,\infty))\right\}.
\end{equation}
}

{
{For a compound Poisson process},  
the expansion (\ref{GnrlExpansion0}) results easily from conditioning on the number of jumps on $[0,t]$.
Thus, infinite-jump activity processes are the interesting cases.}
Ruschendorf and Woerner
\cite{Ruschendorf} (see Theorem 2 in Section 3)
report that for a fixed $N\geq{}1$
and $\eta>0$,
there exists a $\varepsilon'(N)>0$ and $t_{0}>0$ such that,
for all $\varepsilon\in(0,\varepsilon'(N))$ and
$t\in(0,t_{0})$,
\begin{equation}\label{WoernerClaim}
    \bbp(X_{t}\geq {}y)=
    \sum_{i=1}^{N-1} \frac{t^{i}}{i!}\, \nu_{\varepsilon}^{*i}(
    [y,\infty))
    +O_{\varepsilon,\eta}(t^{N}),
    \quad{}for\; y>\eta,
\end{equation}
where
$\nu_{\varepsilon}(dx)={\bf 1}_{\{|x|\geq{}\varepsilon\}}\nu(dx)$.
When $N=3$,
this result would imply that, for $0<\varepsilon<y/2\wedge
\varepsilon'(N)$,
\[
    \bbp(X_{t}\geq{}y)=
    t\,
\nu([y,\infty))
+\frac{t^{2}}{2}
\int_{|u|\geq{}\varepsilon}
\int_{|v|\geq{}\varepsilon}
{\bf 1}_{\{u+v\geq y\}}\nu(dv)\nu(du)
+O_{\varepsilon,\eta}(t^{3}).
\]
{Thus,}  (\ref{WoernerClaim}) would
imply that
\begin{equation*}\label{TargetLimitBad}
\lim_{t\rightarrow{}0}
    \frac{1}{t}\left\{\frac{1}{t}\bbp\left( X_{t}\geq{}y\right)
    -\nu([y,\infty))\right\}=
    \frac{1}{2}
\int_{|u|\geq{}\varepsilon}
\int_{|v|\geq{}\varepsilon}
{\bf 1}_{\{u+v\geq y\}}\nu(dv)\nu(du),
\end{equation*}
which is independent of the Brownian {component $\sigma W_{t}$ and of the ``drift" $bt$}.
We actually found that this limiting value is not the correct
one and provide below the correction using
two different approaches.
Let us point out
where we believe the arguments of \cite{Ruschendorf} are lacking.
The main problem arises from the application of
their Lemma 3 in Theorem 2 (see also Lemma 1 in Theorem 1).
In those lemmas, the value of $t_{0}$ actually depends
on $\delta$. Later on in the proofs,
$\delta$ is taken arbitrarily small,
which is likely to result in $t_{0}\rightarrow{}0$ (unless otherwise
proved).

We prove (\ref{GnrlExpansion0}) using two approaches. 
The first approach is similar in spirit to that in \cite{Ruschendorf}.  
It consists in decomposing the L\'evy process $X$ into two processes, 
one compound Poisson process $\widetilde{X}^{\varepsilon}$ collecting  the ``big'' jumps and another process $X_{t}^{\varepsilon}$ accounting  for the ``small'' jumps.  By conditioning on the number of big jumps during the time interval $[0,t]$, it yields an expression of the form 
\begin{align*}
    \bbp\left(X_{t}\geq{}y\right)
    =
    e^{-\lambda_{\varepsilon}t}\sum_{k=1}^{n}\frac{(\lambda_{\varepsilon}t)^{k}}{k!}
    \bbp\left(X^{\varepsilon}_{t}+\sum_{i=1}^{k}\xi_{i}\geq y \right)+ 
    O(t^{n+1}).
\end{align*}
By taking a compound Poisson process $\widetilde{X}^{\varepsilon}$ with jumps $\{\xi_{i}\}_{i\geq 1}$ having a smooth density, 
one can expand further each term on the right-hand side using the following power series expansion:
\begin{equation}\label{MomentExpansionSmooth}
    \bbe g(X_{t})=g(0)
    +\sum_{k=1}^{n} \frac{t^{k}}{k!}
    L^{k}g(0)+
    \frac{t^{n+1}}{n!}
    \int_{0}^{1} (1-\alpha)^{n}\bbe\left\{
    L^{n+1} g(X_{\alpha t})\right\}d\alpha,
\end{equation}
valid for any $n\geq{}0$ and $g\in C_{b}^{2n+2}$, the
class of functions having continuous and bounded derivatives
of order $0\leq k\leq 2n+2$.  
Above, $L$ is the infinitesimal generator of the L\'evy process, i.e.,  
\begin{equation}\label{InfGen}
    (Lg)(x):=
    \frac{\sigma^{2}}{2}g''(x)+b
    g'(x)+\int
    \left(g(u+x)-g(x)-ug'(x){\bf 1}_{\{|u|\leq{}1\}}\right)
    \nu(du),
\end{equation}
for any function $g\in C_{b}^{2}$.  

{\DRed 
For $n=0$, (\ref{MomentExpansionSmooth}) takes a familiar form
(see e.g. Lemma 19.21 in \cite{Kallenberg}):
\begin{equation}\label{Dynkin}
	\bbe g(X_{t}) =g(0)+t\int_{0}^{1} \bbe\{Lg(X_{\alpha t})\}d\alpha=
	g(0)+\int_{0}^{t} \bbe\{Lg(X_{u})\}d u,
\end{equation}
which is an easy consequence of It\^o's formula. The general case follows easily by induction in $n$. Indeed, if (\ref{MomentExpansionSmooth}) is valid for $n$, applying (\ref{Dynkin}),
$\int_{0}^{1} (1-\alpha)^{n}\bbe\left\{L^{n+1} g(X_{\alpha t})\right\}d\alpha$ becomes
\begin{align*}	
    &\int_{0}^{1} (1-\alpha)^{n}\left\{L^{n+1}g(0)+\alpha t\int_{0}^{1}\bbe\left\{
    L^{n+2} g(X_{\alpha' \alpha t})\right\}d\alpha'\right\}d\alpha\\
    &=\frac{1}{n+1}L^{n+1}g(0)+\frac{t}{n+1} \int_{0}^{1} (1-\hat\alpha)^{n+1}\bbe\left\{
    L^{n+1} g(X_{\hat\alpha t})\right\}d\hat\alpha,
\end{align*}
where we changed variables $\hat\alpha:=\alpha \alpha'$ and applied Fubini' s Theorem. 
Another proof of (\ref{MomentExpansionSmooth}) is given in \cite{Houdre:1998} based on Fourier approximations of $g$ (Proposition 1 and 4 in there). 
}

In the second order approximation case ($n=2$), we give another proof {\DRed for (\ref{GnrlExpansion0})} which relaxes the assumptions on the L\'evy density $s$, by requiring only smoothness in a neighborhood of $y$ and local boundedness away from the origin. This approach is based on the following recent asymptotic result by Jacod \cite{Jacod05}:
\begin{equation}\label{LimitWithsigma}
    \lim_{t\rightarrow{}0}\frac{1}{t}\; \bbe g(X_{t})
    =\sigma^{2}+\int g(x) \nu(dx),
\end{equation}
valid for a $\nu$-continuous bounded function $g$ such that 
$g(x)\sim x^{2}$, as $x\rightarrow{}0$.  In case the process is of finite variation and has no diffusion term, we {\DRed prove} the second order expansion as long as $s$ is continuous at $y$ {\DRed  and locally bounded away from $0$}. We also present a counterexample, originally suggested by Philippe Marchal, which shows that the result is not valid {\DRed if} $s$ is not continuous {\Red(see \cite{Marchal} for further developments)}.

In order to provide explicit formulas for the coefficients $d_{k}$ in (\ref{GnrlExpansion0}), in Section \ref{ApprxmtSect}, we consider a second approach 
whose basic first step is to approximate the indicator 
function ${\bf 1}_{[y,\infty)}$ by smooth functions $f_{m}$ in such a way that 
\[
	\lim_{m\rightarrow\infty}\bbe f_{m}(X_{t})=\bbp(X_{t}\geq{}y).
\] 
The idea to derive (\ref{GnrlExpansion0}) is to apply (\ref{MomentExpansionSmooth}) to each smooth 
approximation $f_{m}$ and show that the limit of each term in the power expansion converges as $m\rightarrow\infty$. We emphasize that this approach is carried out without additional assumption on $s$, except smoothness and local boundedness away from the origin.
In Section \ref{DnstyExp},
we exploit further the approximation of $f$ by the smooth functions $f_{m}$ to provide an explicit formula for the remainder $\calR_{n}(t)$ in (\ref{GnrlExpansion0}).
To carry out this plan, {we impose more stringent conditions on $X$} than those required in the first approach.  
In particular, we require that $X_{t}$ has a $C^{\infty}$-transition density $p_{t}$, whose derivatives 
remain uniformly bounded away from the origin, as $t\rightarrow{}0$.  
As a byproduct of the explicit remainder, polynomial expansions of order $n$ in $t$ are derived for the transition densities of the process extending a result in \cite{Ruschendorf}.

{
In Section \ref{SectSymmtLevy}, the boundedness conditions on the derivatives of the transition densities are shown to hold for 
symmetric stable L\'evy processes. The validity of this uniform boundedness for general \emph{tempered stable processes} is also considered in Section \ref{SectGnrlSymmtLevy}, via a recursive formula for the derivatives of the transition density. {Tempered stable processes have received a great dealt of attention in the last decade due to their applications in mathematical finance. Among their members, we can list the CGMY model of \cite{Madan}. See Rosi\'nski \cite{Rosinski:2007} for a detailed study of this class of processes.}

We note finally, that throughout the paper we only consider asymptotics for $\bbp(X_t\ge y), y >0$, but that 
our methodology also gives results for $\bbp(X_t\le -y), y >0$, replacing $\nu([y, +\infty))$ by 
$\nu((-\infty, -y])$.  

{
\section{An application: nonparametric estimation of the L\'evy density}\label{Motivation}
In this part we present an application of the small-time asymptotics considered in this work as a matter of motivation. One problem that has received attention in recent years is that of estimating the L\'evy density $s$ of the process in a non-parametric fashion. This means that, by only imposing qualitative constraints on the L\'evy density (e.g. smoothness, monotonicity, etc.), we aim at constructing a function $\hat{s}$ that is consistent with the available observations  of the process $X$. The minimal desirable requirement of our estimator $\hat{s}$ is consistency;
namely, the convergence
\(
	\hat{s}\rightarrow s,
\)
say in a mean-square error sense,  must be ensured when the available sample of the process {increases}.

 When the data available consists of the \emph{whole trajectory} of the process during a time interval $[0,T]$, the problem is equivalent to estimating the intensity function of an inhomogeneous Poisson process {(see e.g.  \cite{Reynaud} for the case of finite intensity functions and \cite{FigHou:2006} for the case of L\'evy processes, where the intensity function could be infinite)}. However, {a continuous-time sampling} is not feasible in reality, and thus, the relevant problem is that of estimating $s$ based on discrete sample data $X_{t_{0}},\dots, X_{t_{n}}$ during a time interval $[0,T]$. In that case, the jumps are latent variables whose statistical properties can {in principle}  be assessed {if} the frequency and time horizon of observations increase to infinity. 
 
 It turns out that asymptotic results such as (\ref{TargetLimit}) and (\ref{GnrlExpansion0}) play important roles in determining how {frequently} one should sample (given the time horizon $T$ at hand) such that the resulting discrete sample contains sufficient information about the whole path.  We can say that a given discrete sample scheme is good enough if we can devise a discrete-based estimator for the parameter of interest that enjoys a rate of convergence comparable to that of a good continuous-based estimator.
 Let us explain this point with a concrete example. Consider the estimation of the following functional of $s$:
 \[
 	\beta(\varphi):=\int \varphi(x) s(x)dx, 
\]
where $\varphi$ is a function that is smooth \emph{on its support}. Assume also that the support of $\varphi$ is an interval $[c,d]$ so that the indicator ${\bf 1}_{[c,d]}$ vanishes in a neighborhood of the origin. A natural continuous-based estimator of $\beta(\varphi)$ is given by 
\[
	\beta^{c}_{_{T}}(\varphi):=\frac{1}{T}\sum_{s\leq{}T} \varphi(\Delta X_{s}).
\]
Using the well-known formulas for the mean 
and variance of Poisson integrals (see e.g. \cite[Proposition 19.5]{Sato}),
the above estimator can be seen to converge to $\beta(\varphi)$, and moreover,
\[
	\bbe (\beta^{c}_{_{T}}(\varphi)-\beta(\varphi))^{2}=\frac{1}{T}\,\beta(\varphi^{2}).
\]
{{}We can thus} say that $\beta^{c}_{_{T}}(\varphi)$ converges to $\beta(\varphi)$  at the rate of {\DRed $O(T^{-1/2})$, in the mean-square sense}. 

Suppose that {{}instead we} use a reasonable discrete-based proxy of $\beta^{c}_{_{T}}$, using the increments $X_{t_{1}}-X_{t_{0}},\dots,X_{t_{n}}-X_{t_{n-1}}$  of the process instead of the jumps $\Delta X_{t}$:
\[
	\beta^{\pi}_{_{T}}(\varphi):=\frac{1}{T}\sum_{i=1}^{n} \varphi( X_{t_{i}}-X_{t_{i-1}}),
\]
where {{}$\pi:t_{0}<\dots<t_{n}=T$}. A natural question is then the following: How frequently  should the process be sampled so that $\beta^{\pi}_{_{T}}(\varphi)\rightarrow\beta(\varphi)$ at a rate of {{}{\DRed $O(T^{-1/2})$}?} To show in a simple manner the connection between the previous question and the asymptotics  (\ref{TargetLimit}), suppose that the sampling is ``regular" in time with fixed time span $\Delta_{n}:=T/n$ between consecutive observations.  In that case, 
we have 
\begin{align*}
	\bbe (\beta^{\pi}_{_{T}}(\varphi)-\beta(\varphi))^{2}&\leq \frac{1}{T}\beta(\varphi^{2}) + 
	\frac{1}{T}\left\{\frac{1}{\Delta_{n}}\bbe {{}\varphi^{2}\left(X_{\Delta_{n}}\right)}-\beta(\varphi^{2})\right\}\\
	&\quad+
	\left\{\frac{1}{\Delta_{n}}\bbe \varphi\left(X_{\Delta_{n}}\right)-\beta(\varphi)\right\}^{2}.
\end{align*}
From the previous inequality we see that the rate of convergence in the limit
\begin{equation}\label{CnvMmntFnct}
	\lim_{\Delta\rightarrow{}0}
	\frac{1}{\Delta}\bbe \varphi\left(X_{\Delta}\right)=\beta(\varphi), 
\end{equation} 
will determine the rate of convergence of $\beta^{\pi}_{_{T}}(\varphi)$ towards $\beta(\varphi)$.
To determine the rate of convergence in (\ref{CnvMmntFnct}), one can simply  link  $\bbe \varphi\left(X_{\Delta}\right)$ to $\bbp(X_{\Delta}\geq y)$, and link
$\beta(\varphi)$ to $\nu([y,\infty))$. This is easy if $\varphi$ is smooth {{}on its support} $[c,d]$. Indeed, we have that
     \[
        \left|\frac{\bbe \varphi\left(X_{\Delta}\right)}{\Delta}
        -\beta(\varphi)\right|\leq
        \left(\|\varphi\|_{\infty}+\|\varphi'\|_{1}\right)
        \sup_{y\in[c,d]}
     \left|\frac{1}{\Delta}\bbp\left[X_{\Delta}\geq{}y\right]-\nu([y,\infty))\right|.
     \]
Hence, the rate of convergence of $\Delta^{-1}\bbp(X_{\Delta}\geq y)$ towards $\nu([y,\infty))$ determines the rate of convergence of $\Delta^{-1}\bbe \varphi\left(X_{\Delta}\right)$ towards $\beta(\varphi)$. In particular, the result  (\ref{TargetLimit}) will tell us that, for  $\beta^{\pi}_{_{T}}(\varphi)$ to converge to $\beta(\varphi)$ at a rate of {\DRed $O(T^{-1/2})$, in the mean-square sense}, it suffices that {\DRed the time span between consecutive observations $\Delta$} is $o(T^{-1/2})$. It is important to remark that (\ref{TargetLimit}) can be {{}seen} to hold uniformly {{}in} {\DRed $y>\underline{y}$, for an arbitrary $\underline{y}>0$}. 

The ideas outlined in this section, as well as the asymptotic result  (\ref{TargetLimit}), 
are heavily exploited in \cite{Fig:2007a} and \cite{Fig:2008a}, where the general problem of nonparametric estimation of the L\'evy density $s$ is studied {{}using Grenander's method of sieves}.
}

\section{Expansions for the transition distribution}\label{Results}

As often, e.g. see \cite{Ruschendorf}, the general strategy
is to decompose the L\'evy
process into two processes: one accounting for the ``small'' jumps and
a compound Poisson process collecting the ``big'' jumps.
Concretely, suppose that $X$ has L\'evy triplet
$(\sigma^{2},b,\nu)$; that is, $X$ admits the decomposition
\begin{equation}\label{LevyItoDecmp}
    X_{t}= bt+\sigma W_{t}+\int_{0}^{t}\int_{|x|\leq{}1}
    x\,
    (\mu-\bar\mu)(dx,ds)+
    \int_{0}^{t}\int_{|x|>1} x\,
    \mu(dx,ds),
\end{equation}
where $W$ is a standard Brownian motion and
$\mu$ is an independent Poisson measure on
$\bbr_{+}\times\bbr\backslash\{0\}$ with
mean measure $\bar\mu(dx,dt):= \nu(dx)dt$.
Note that $\mu$ is the random measure associated to 
the jumps of $X$. 
Given a smooth truncation function $c_{\varepsilon}\in C^{\infty}$ such that 
${\bf 1}_{[\varepsilon/2,\varepsilon/2]}(x)
\leq c_{\varepsilon}(x) \leq  {\bf 1}_{[\varepsilon,\varepsilon]}(x)$, set
\begin{align}\label{TrctedLevy}
    \widetilde{X}^{\varepsilon}_{t}&:=
    \int_{0}^{t}\int_{\bbr} x \, \bar{c}_{\varepsilon}(x)
    \mu(dx,ds),\\
    X^{\varepsilon}_{t} &:= X_{t}- \widetilde{X}^{\varepsilon}_{t}, \label{Rmdr}
\end{align}
where $\bar{c}_{\varepsilon}(x):=1-c_{\varepsilon}$.
It is well-known that  $\widetilde{X}^{\varepsilon}$ is a 
compound Poisson process with intensity of jumps
$\lambda_{\varepsilon}:=\int \bar{c}_{\varepsilon}(x)\nu(dx)$,
and jumps distribution 
$ \bar{c}_{\varepsilon}(x)\nu(dx)/\lambda_{\varepsilon}$.
The remaining process  $X^{\varepsilon}$ is then a L\'evy process with jumps bounded by
$\varepsilon$ and L\'evy triplet 
$(\sigma^{2},b_{\varepsilon},
{c}_{\varepsilon}(x) \nu(dx))$,
where
\[
    b_{\varepsilon}:=b-\int_{|x|\leq{}1}x \bar{c}_{\varepsilon}(x)\nu(dx).
\]

There are two key results that will be used to arrive to (\ref{GnrlExpansion0}). 
The first is the expansion (\ref{MomentExpansionSmooth}). The following tail estimate will also play an important role in the sequel: 
\begin{equation}\label{TailEstm}
        \bbp\left(\left|X^{\varepsilon}_{t}\right|\geq  y\right)
        \leq{} \exp\{a y_{0}\log y_{0}\} \exp\left\{ay-ay\log y\right\}
        t^{ya},
\end{equation}
valid for an arbitrary, but fixed, positive real $a$ in $(0,\varepsilon^{-1})$,
and for any {\DRed $t,y>0$ such that $t< y_{0}^{-1}y$, where $y_{0}$ depends only upon $a$
(see \cite[Lemma 3.2]{Ruschendorf} or \cite[Section 26]{Sato}
for a proof).}

\begin{remk} 
For an alternative proof of (\ref{TailEstm}), 
use a generic concentration inequality such as \cite[Corollary 1]{Houdre:2002} 
to get (when $\sigma = 0$):  
\begin{align*}
    \bbp(X_t^\varepsilon \geq y)&=\bbp(X_t^\varepsilon - \bbe X_t^\varepsilon \geq x) \\
    &\leq e^{-\frac{x}{\varepsilon}+\left(\frac{x}{\varepsilon}+\frac{tV^2}{\varepsilon^2}\right)
\log\left(1+\frac{\varepsilon x}{tV^2}\right)} 
    \leq \left(\frac{eV^2}{\varepsilon x}\right)^{\frac{x}{\varepsilon}}t^{\frac{x}{\varepsilon}},
\end{align*}
whenever $x:=y - \bbe X_t^\varepsilon > 0$, and with $V^2 := \int_{|u|\leq \varepsilon}u^2\nu(du)$.    
Now $\bbe X_{t}^{\varepsilon}=t(b_{\varepsilon}+\int_{\{1< |x|\leq \varepsilon\}} x\nu(dx))$, and as $t\to 0$, $x\to y$ and $(eV^2/\varepsilon x)^{x/\varepsilon} \to (eV^2/\varepsilon y)^{y/\varepsilon}$, 
with moreover $t^{x/\varepsilon}/t^2 = \exp((y-\bbe X_t^\varepsilon -2\varepsilon) \log t/\varepsilon) \to 0$, 
as long as $y> 2\varepsilon$.  Finally, since as $t\to 0$, 
$\bbp(\sigma W_t \geq y/2)/t^2 \to 0$, the general case follows.   
\end{remk}

We are ready to show (\ref{GnrlExpansion0}). Below, $L_{\varepsilon}$ is the infinitesimal generator of $X^{\varepsilon}$ and we use the following notation:
\begin{align*}
	{s}_{\varepsilon}:=c_{\varepsilon} s,\quad
	\bar{s}_{\varepsilon}:=1- s_{\varepsilon}, \quad 
	&L_{\varepsilon}^{0} g =g,\quad  \bar{s}_{\varepsilon}^{*1}=\bar{s}_{\varepsilon}\\
	\bar{s}_{\varepsilon}^{*i}(x)=\int  \bar{s}_{\varepsilon}^{*(i-1)}(x-u)\bar{s}_{\varepsilon}(u)du,&\quad(i\geq 2), \quad \bar{s}_{\varepsilon}^{*0}*g = g.
\end{align*}
\begin{thrm}\label{MainExpansion}
    {\DRed Let $\underline{y}>0$,  $n\geq 1$, and $0 < \varepsilon<\underline{y}/(n+1)\wedge{}1$}.
    Assume that $\nu$ has a  density $s$ such that  for any $0\leq k\leq{} 2n+1$ and any $\delta>0$,
    \[
        a_{k,\delta}:=\sup_{|x|>\delta}|s^{(k)}(x)|<\infty.
    \]
    Then, {\DRed there exists a $t_{0}>0$ such that,  for 
    any $y\geq \underline{y}$ and $0<t<t_{0}$},
\begin{align}\label{FExp1}
    {\DRed \bbp\left(X_{t}\geq{}y\right)
    =e^{-\lambda_{\varepsilon}t}\sum_{j=1}^{n}c_{j}\, \frac{t^{j}}{j!}
    +O_{\varepsilon,\underline{y}}(t^{n+1})},
\end{align}
where 
\[
	c_{j}:=   
    \sum_{i=1}^{j} \binom{j}{i} L_{\varepsilon}^{j-i}\hat{f}_{i}(0),
\]
with $\hat{f}_{i}(x):=\int_{y-x}^{\infty} \bar{s}_{\varepsilon}^{*i}(u) du$.
\end{thrm}
\begin{proof}
{\DRed Throughout this part, we write $f(x):={\bf 1}_{\{x\geq y\}}$}.
In terms of the decomposition $X:=X^{\varepsilon}+\widetilde{X}^{\varepsilon}$ described at the beginning of this section,  by conditioning on the number of jumps of $\widetilde{X}^{\varepsilon}$ during the interval $[0,t]$,
we have that 
\begin{align}\label{T1}
    \bbe f(X_{t})&
    =\bbe f\left(X^{\varepsilon}_{t}\right)
    e^{-\lambda_{\varepsilon}t}+
    e^{-\lambda_{\varepsilon}t}\sum_{k=n+1}^{\infty}\frac{(\lambda_{\varepsilon}t)^{k}}{k!}
    \bbe f\left(X^{\varepsilon}_{t}+\sum_{i=1}^{k}\xi_{i}\right)\\
    \label{T2}
    &\quad+
    e^{-\lambda_{\varepsilon}t}\sum_{k=1}^{n}\frac{(\lambda_{\varepsilon}t)^{k}}{k!}
    \bbe f\left(X^{\varepsilon}_{t}+\sum_{i=1}^{k}\xi_{i}\right)
\end{align}
where 
$\xi_{i}\stackrel{iid}{\sim} \bar{c}_{\varepsilon}(x) s(x)dx/\lambda_{\varepsilon}$.
Taking {\DRed $a:=(n+1)/\underline{y}$}, (\ref{TailEstm}) and  $0\leq f\leq 1$ imply that  the two terms on the right hand side of (\ref{T1}) are {\DRed $O_{\varepsilon,\underline{y}}(t^{n+1})$} as $t\rightarrow 0$, {\DRed provided that $t<t_{0}:=y_{0}^{-1}\underline{y}$}. Next, for each $k\geq 1$, 
\[
	\bbe f\left(X^{\varepsilon}_{t}+\sum_{i=1}^{k}\xi_{i}\right)=\bbe \widetilde{f}_{k}
	\left(X^{\varepsilon}_{t}\right),
\]
where
\[
	 \widetilde{f}_{k}(x):=\bbe  f\left(x+\sum_{\ell=1}^{k}\xi_{i}\right)
	 =\bbp\left(x+\sum_{\ell =1}^{k}\xi_{i}\geq y\right),
\]
which is $C^{2n+2}_{b}$, since the density of $\xi_{i}$ is $C_{b}^{2n+1}$. 
Then, one can apply (\ref{MomentExpansionSmooth}) to get 
\begin{equation}\label{AE1}
    \bbe \widetilde{f}_{k}(X^{\varepsilon}_{t})= 
    \sum_{i=0}^{n-k} \frac{t^{i}}{i!}
    L^{i}_{\varepsilon}\widetilde{f}_{k}(0)+
    \frac{t^{n+1-k}}{(n-k)!}
    \int_{0}^{1} (1-\alpha)^{n-k}\bbe\left\{
    L_{\varepsilon}^{n+1-k} \widetilde{f}_{k}(X^{\varepsilon}_{\alpha t})\right\}d\alpha.
\end{equation}
Let $L_{\varepsilon}$ be the infinitesimal generator of $X^{\varepsilon}$, given by
\begin{align*}
    (L_{\varepsilon}g)(x)&
    =b_{\varepsilon}
    g'(x)+
    \frac{\sigma^{2}}{2}g''(x)+\int\int_{0}^{1} g''(x+\beta w) (1-\beta) d\beta w^{2}c_{\varepsilon}(w) s(w) dw,
\end{align*}
for $g\in C^{2}_{b}$, and for $k\geq{}1$, let 
\[
	d\pi_{k}^{\varepsilon}:= \Pi_{\ell=1}^{k} (1-\beta_{\ell})d\beta_{\ell} w_{\ell}^{2}c_{\varepsilon}(w_{\ell}) s(w_{\ell}) dw_{\ell},
\]
which clearly a finite measure on $[0,1]^{k}\times \bbr^{k}$. Then, note that
\begin{equation}\label{DcmL}
	(L_{\varepsilon}^{i}g)(x)=
	\sum_{{\bf k}\in\mathcal{K}_{i}}
        c_{\bf k}\binom{i}{\bf k}  A^{\varepsilon}_{{\bf k}} g (x),
\end{equation}
where $\mathcal{K}_{i}:=\{{\bf k}:=(k_{1},k_{2},k_{3}):k_{1}+k_{2}+k_{3}=i\}$,
\begin{align*}
   c_{\bf k}&:= b_{\varepsilon}^{k_{1}}\left\{\sigma^{2}/2\right\}^{k_{2}},\\
    A^{\varepsilon}_{{\bf k}} g (x)&:=
    \int g^{(k_{1}+2k_{2}+2k_{3})}
    \left(x+
        \displaystyle{\sum_{\ell=1}^{k_{3}}}\beta_{\ell}w_{\ell}
    \right)d\pi_{k_{3}}^{\varepsilon},
\end{align*}
if $k_{3}\geq 1$ and $A_{{\bf k}} g (x):= g^{(k_{1}+2k_{2})}(x)$, if $k_{3}=0$.
Since 
\[
	\widetilde{f}_{k}^{(\ell)}(x)=\lambda_{\varepsilon}^{-k} (-1)^{\ell-1}\bar{s}_{\varepsilon}^{*(k-1)}* 	\bar{s}_{\varepsilon}^{(\ell-1)}(y-x), 
\]
and $\bar{s}_{\varepsilon}(\cdot)\in C^{2n+1}_{b}$,
there exists a constant 
$b_{n,\varepsilon}<\infty$ {\DRed (independent of $y$)}, such that
\[
	\|L_{\varepsilon}^{n+1-k} \widetilde{f}_{k}\|_{\infty}\leq b_{n,\varepsilon}
	(a_{2n+1,\varepsilon/2})
\]
and so, the last term in (\ref{AE1}) is $O(t^{n+1-k})$. Plugging (\ref{AE1}) into (\ref{T1}) and rearranging terms, we get 
\begin{align*}
    \bbe f(X_{t})&
    =e^{-\lambda_{\varepsilon}t}\sum_{j=1}^{n}\frac{t^{j}}{j!}
    \sum_{k=1}^{m} \binom{m}{k} \lambda_{\varepsilon}^{k}L_{\varepsilon}^{m-k}\widetilde{f}_{k}(0)
    +{\DRed O_{\varepsilon,\underline{y}}(t^{n+1})},
\end{align*}
which is exactly (\ref{FExp1}), because $\lambda^{k}_{\varepsilon}L_{\varepsilon}^{m-k}\widetilde{f}_{k}=L_{\varepsilon}^{m-k}\hat{f}_{k}$.
\end{proof}

\begin{remk}\label{Simplify}\hfill
\begin{enumerate}
\item[(i)] The expansion (\ref{GnrlExpansion0}) follows from (\ref{FExp1}). Indeed, expanding $e^{-\lambda_{\varepsilon}t}$, we get that {\DRed for 
    any $y\geq \underline{y}$ and $0<t<t_{0}$}:
\begin{equation}\label{GnrlExpansion1}
        \bbp\left(X_{t}\geq{}y\right)=
        \sum_{k=1}^{n} d_{k}\,\frac{t^{k}}{k!}+{\DRed O_{\varepsilon,\underline{y}}(t^{n+1})},
    \end{equation}
with 
\begin{equation}\label{Cnst1}
	d_{k}=\sum_{j=1}^{k}  \binom{k}{j} c_{j} (-\lambda_{\varepsilon})^{k-j}.
\end{equation}
In the next section we give a more explicit expression for $d_{k}$. 

\item[(ii)] The first two terms in (\ref{GnrlExpansion1}) can be  easily computed:
\begin{align*}
	d_{1}&= \int_{y}^{\infty} s(u) du =\nu([y,\infty))\\
	d_{2}&= - 2\lambda_{\varepsilon} \nu([y,\infty))
	+ \iint {\bf 1}_{\{u_{1}+u_{2}\geq{}y\}}
        \bar{s}_{\varepsilon}(u_{1})\bar{s}_{\varepsilon}(u_{2}) du_{1} du_{2}\\
        &\quad-\sigma^{2}s'(y)
        +2b_{\varepsilon}s(y)-2
        \int \int_{0}^{1}s'(y-\beta w)
        (1-\beta)d\beta  w^{2} s_{\varepsilon}(w)d w.\
\end{align*}

\item[(iii)]
The coefficients $d_{k}$ in (\ref{GnrlExpansion1}) are independent of $\varepsilon$ since they can be defined iteratively as limits of $ \bbp\left(X_{t}\geq{}y\right)$. 
For instance, 
\begin{align*}
	\lim_{t\rightarrow{}0}\frac{1}{t}\bbp\left( X_{t}\geq{}y\right)=d_{1},\quad
	\lim_{t\rightarrow{}0}
    \frac{1}{t}\left\{\frac{1}{t}\bbp\left( X_{t}\geq{}y\right)
    -d_{1}\right\}=d_{2}.
\end{align*}
One can obtain an expression for $d_{2}$ that is independent of $\varepsilon$ by taking the 
limit as $\varepsilon\rightarrow{}0$. For instance, if $X$ is of bounded variation with drift $b_{0}:=b-\int_{|x|\leq{}1} x\nu(dx)$ and volatility $\sigma$, then $d_{2}$ becomes 
\begin{align*}
	d_{2}&=-\sigma^{2} s'(y) +2 b_{0} s(y) -(\nu([y,\infty)))^{2}\\
	&\quad+\int_{0}^{y} \int_{y-x}^{y} s(u)du s(x) dx
	+2\int_{y}^{\infty}\int_{y-x}^{0} s(u)du s(x) dx.	
\end{align*}
In general,  it turns out (see the Appendix) that $d_{2}$  ``simplifies" to the following expression when $\varepsilon\rightarrow{}0$:
{\DRed
\begin{align*}
d_{2}&=-\sigma^{2}\,s'(y)+  2b s(y)-
    \nu((y,\infty))^{2}
    +\nu((y/2,y))^{2}\\
    &\quad
    +2\int_{-\infty}^{-y/2}
    \int_{y-x}^{y} s(u) du s(x) dx
    -2s(y)\int_{y/2<|x|\leq{}1} x s(x) d x\\
    &\quad
    + 2  \int_{-y/2}^{y/2}
    \int_{y-x}^{y} \left\{s(u)-s(y)\right\} du s(x) dx. \quad \Box
\end{align*}
}
\end{enumerate}
\end{remk}

We now present an alternative proof for the expansion (\ref{TargetLimit1}) that requires less stringent assumptions. The following asymptotic result due to Jacod \cite{Jacod05} will be of importance:
\begin{equation}\label{LimitWithsigma}
    \lim_{t\rightarrow{}0}\frac{1}{t}\; \bbe g(X_{t})
    =\sigma^{2}+\int g(x) \nu(dx),
\end{equation}
valid if $g$ is $\nu$-continuous, bounded, and  such that 
$g(x)\sim x^{2}$, as $x\rightarrow{}0$.
{
\begin{prop}\label{2ndExpansion}
    Let $y>0$ and $0 < \varepsilon<y/2\wedge{}1$.
    Assume that $\nu$ has a  density $s$ which is 
    bounded outside of the interval 
    $[-\varepsilon,\varepsilon]$, and that
    is $C^{1}$ in a neighborhood of $y$. Then, the limit (\ref{TargetLimit1}) exists and can be written as:
\begin{align*}
    {\frac{d_{2}}{2}}&=
    -\frac{\sigma^{2}}{2}\,s'(y)+  b_{\varepsilon}\, s(y)
        +\int
        \int_{0}^{x}\left\{s(y-u)-s(y)\right\}du  s_{\varepsilon}(x) dx\\
    &\quad+\frac{1}{2} \int\int
    {\bf 1}_{\{x+u\geq{}y\}} \bar{s}_{\varepsilon}(u) \bar{s}_{\varepsilon}(x)du dx- \lambda_{\varepsilon}\nu([y,\infty)).
\end{align*}
\end{prop}
}
\begin{proof}
Let $f(x):={\bf 1}_{\{x\geq{}y\}}$ and let
\begin{equation*}
    A(t):=
    \frac{1}{t}\left\{
    \frac{1}{t}\; \bbe f(X_{t})
    -\int f(x) \nu(dx)\right\}.
\end{equation*} 
With the notation of Theorem  \ref{MainExpansion}, we have
\begin{align}\label{AuxEq0}
    A(t)&=\frac{1}{t^{2}}\bbe f\left(X^{\varepsilon}_{t}\right)
    e^{-\lambda_{\varepsilon}t}+
    e^{-\lambda_{\varepsilon}t}
    \int\frac{1}{t}\left\{
    \bbe f(X_{t}^{\varepsilon}+x)-f(x)\right\}\bar{s}_{\varepsilon}(x) dx\\
    &-\frac{1-e^{-\lambda_{\varepsilon} t}}{t}
    \int f(x) \bar{s}_{\varepsilon}(x) dx+
    e^{-\lambda_{\varepsilon} t}\sum_{n=2}^{\infty}\frac{(\lambda_{\varepsilon})^{n}
    t^{n-2}}{n!}
    \bbe f\left(X^{\varepsilon}_{t}+\sum_{i=1}^{n}\xi_{i}\right),\nonumber
\end{align}
since $\varepsilon<y/2$.
In view of (\ref{TailEstm}), the first term on the right hand side vanishes when $t\rightarrow 0$.
 Then, except for the second term, all the other terms are
easily seen to be convergent. Let us thus analyze the second term. Let
    \[
        B(t):=\int \left\{
        \bbe f(X_{t}^{\varepsilon}+x)-f(x)\right\}\bar{s}_{\varepsilon}(x) dx.
    \]
Since $0<\varepsilon<y/2$ and the support of $c_{\varepsilon}$ is $[-\varepsilon,\varepsilon]$, 
$B(t)$ can be decomposed as
\begin{align*}
    B(t)&:= \int_{y-\varepsilon}^{y}
    \bbp\left(X_{t}^{\varepsilon}\geq{}y-x\right) {s}(x) dx
    -\int_{y}^{y+\varepsilon}
    \bbp\left(X_{t}^{\varepsilon}< y-x\right) {s}(x) dx\\
    &\quad+\int_{x<y-\varepsilon}
    \bbp\left(X_{t}^{\varepsilon}\geq{}y-x\right)\bar{s}_{\varepsilon}(x)  dx
    -\int_{y+\varepsilon}^{\infty}
    \bbp\left\{X_{t}^{\varepsilon}< y-x\right\} \bar{s}_{\varepsilon}(x) dx.
\end{align*}
Since $s$ is bounded and integrable away from the origin,
the last two terms can be upper bounded by
\(
    \lambda_{\varepsilon}
    \bbp\left\{|X_{t}^{\varepsilon}|>\varepsilon\right\},
\)
which, divided by $t$, converges to $0$ in view of (\ref{FrstLimit}).
After changing variables to $u=y-x$
and applying Fubini's Theorem,
the first term above becomes:
\begin{align*}
    \int_{y-\varepsilon}^{y}
    \bbp\left(X_{t}^{\varepsilon}\geq{}y-x\right) {s}(x) dx
    = \int_{0}^{\varepsilon}
    \bbp\left(X_{t}^{\varepsilon}\geq{}u\right) s(y-u)du
    =\bbe f_{+}\left(X_{t}^{\varepsilon}\right),
\end{align*}
where
\(
    f_{+}(x):=\int_{0}^{(x\wedge\varepsilon)\vee 0} s(y-u) du.
    \)
Similarly,
\begin{align*}
    \int_{y}^{y+\varepsilon}
    \bbp\left(X_{t}^{\varepsilon}<y-x\right) s(x)dx
    = \int_{0}^{\varepsilon}
    \bbp\left(X_{t}^{\varepsilon}<-u\right) s(y+u)du
    =\bbe f_{-}\left(X_{t}^{\varepsilon}\right),
\end{align*}
where
\(
    f_{-}(x):=\int_{0}^{(-x\wedge\varepsilon)\vee 0} s(y+u) du.
    \)
Next, consider the function
\begin{align*}
    \widetilde{f}(x)&:=
    \left\{
        \begin{array}{ll}
        f_{+}(x)-
    s(y)\,(x\wedge{}\varepsilon),
        & x>0\\
        \\
    -f_{-}(x)+
    s(y)\,(-x\wedge{}\varepsilon),
        & x<0, 
    \end{array}\right.
\end{align*}
and note that
\(
    \lim_{x\rightarrow{}0}\widetilde{f}(x)/x^{2}
    =- s'(y)/2.
\)
In view of  (\ref{LimitWithsigma}),
we conclude that 
\begin{align*}
    \lim_{t\rightarrow{}0}\frac{1}{t}
     \bbe\widetilde{f}\left(X_{t}^{\varepsilon}\right)
    &=-\frac{s'(y)}{2}\,\sigma^{2}+
    \int
    \widetilde{f}(x) c_{\varepsilon}(x) s(x) dx.
\end{align*}
Thus,  the sum of the first two terms in the decomposition of $B(t)$ are 
\begin{align}\label{AuxEq}
    \bbe f_{+}\left(X_{t}^{\varepsilon}\right)
    -\bbe f_{-}\left(X_{t}^{\varepsilon}\right)&=
    \bbe \widetilde{f}\left(X_{t}^{\varepsilon}\right)
    +s(y)\bbe
    h(X_{t}^{\varepsilon}),
\end{align}
where $h(x)=x{\bf 1}_{|x|\leq{}\varepsilon} -\varepsilon {\bf 1}_{x<-\varepsilon}+
\varepsilon {\bf 1}_{x>\varepsilon}$. 
Let us analyze the last term in (\ref{AuxEq}):
\begin{align*}
    \lim_{t\rightarrow{}0}\frac{1}{t}\bbe
    h(X_{t}^{\varepsilon})&=
    \lim_{t\rightarrow{}0}\frac{1}{t}\bbe
    X_{t}^{\varepsilon}
    -\lim_{t\rightarrow{}0}\frac{1}{t}\bbe
    X_{t}^{\varepsilon}\,{\bf 1}_{\{|X_{t}^{\varepsilon}|>\varepsilon\}}\\
    \quad&+\varepsilon\lim_{t\rightarrow{}0}\frac{1}{t}
    \bbp\{X_{t}^{\varepsilon}>\varepsilon\}
    -\varepsilon\lim_{t\rightarrow{}0}\frac{1}{t}
    \bbp\{X_{t}^{\varepsilon}<-\varepsilon\}=b_{\varepsilon}.
\end{align*}
We are finally able to give the limit of
$B(t)/t$:
    \begin{align*}
        \lim_{t\rightarrow{}0} \frac{1}{t}B(t)
        &= -\frac{s'(y)}{2}\,\sigma^{2}+ s(y)\, b_{\varepsilon}
        +\int \widetilde{f}(x) {c}_{\varepsilon}(x) s(x) dx.
    \end{align*}
A little extra work leads to the expression in the statement of the result.
\end{proof}

{
    It is not clear whether or not
   Proposition \ref{2ndExpansion} remains
    true when $\sigma\neq {}0$ and  the density of $\nu$ is not 
    differential in a neighborhood of $y$.
    If $\sigma=0$, one can relax the differentiability condition as follows.
\begin{prop}\label{PrJmpCs}
    Let $y>0$ and $0 < \varepsilon<y/2\wedge{}1$.
    Assume that $\nu$ has a density $s$ which is 
    bounded outside of the interval 
    $[-\varepsilon,\varepsilon]$ and that
    is continuous in a neighborhood of $y$. Assume also that $\sigma=0$ and that 
    \begin{equation}\label{FntVarCnd}
     	\int_{\{|x|\leq{}1\}}|x|\nu(dx)<\infty.
    \end{equation}
     Then, the limit (\ref{TargetLimit}) exists and is given by
\begin{align*}
    {\frac{d_{2}}{2}}&:=  b_{0}\, s(y)
        +\int
        \int_{0}^{x}s(y-u)du s_{\varepsilon}(x) dx\\
        &\quad+\frac{1}{2} \int\int
    {\bf 1}_{\{x+u\geq{}y\}} \bar{s}_{\varepsilon}(u) \bar{s}_{\varepsilon}(x)du dx- \lambda_{\varepsilon}\nu([y,\infty)),
\end{align*}
where $b_{0}:=b-\int_{|x|\leq{}1}x\nu(dx)$.
\end{prop} 
\begin{proof}
	The proof is very similar to that of Proposition \ref{2ndExpansion}. However, instead of 
	(\ref{LimitWithsigma}), we use the following asymptotic result 
\begin{equation}\label{LimitWithoutsigma}
    \lim_{t\rightarrow{}0}\frac{1}{t}\; \bbe g(X_{t})
    =|b_{0}|+\int g(x) \nu(dx),
\end{equation}
valid for any continuous bounded function $g$ such that 
$g(x)\sim |x|$, as $x\rightarrow{}0$ (see e.g. Jacod \cite{Jacod05}).
Define the function 
\begin{align*}
    \widehat{f}(x)&:=
    \left\{
        \begin{array}{ll}
        f_{+}(x),
        & x>0,\\
        \\
    -f_{-}(x),
        & x<0, 
    \end{array}\right.
\end{align*}
and note that
\(
    \lim_{x\rightarrow{}0}\widehat{f}(x)/x
    =s(y).
\)
By (\ref{LimitWithoutsigma}),
\begin{align}\label{LmtLnrGrwth}
    \lim_{t\rightarrow{}0}\frac{1}{t}
     \bbe\widehat{f}\left(X_{t}^{\varepsilon}\right)
    &=s(y)b_{0}+
    \int_{|x|\leq{}\varepsilon}
    \widehat{f}(x) \nu(dx).
\end{align}
Using the arguments of the proof of Proposition \ref{2ndExpansion}, 
(\ref{LmtLnrGrwth}) implies that 
\begin{equation*}
        \lim_{t\rightarrow{}0} \frac{1}{t}B(t)
        = b_{0}\, s(y)
        +\int
        \int_{0}^{x}s(y-u) s_{\varepsilon}(x) dx,
    \end{equation*}
and this gives the value of $d_{2}$ stated in the statement of the result.
\end{proof}

\begin{remk}
	The continuity of $s$ is needed in Proposition \ref{PrJmpCs}, as the following example suggested by Philippe Marchal shows (see \cite{Marchal}, for further developments). Let $X_{t}:=S_{t}+Y_{t}$, where $S$ is a strictly $\alpha$-stable L\'evy process such that $S_{t}\ed t^{1/\alpha} S_{1}$ and $Y$ is an independent compound Poisson with jumps $\{\xi_{i}\}_{i}$ and jump intensity $1$. 
Suppose that $1<\alpha<2$ and that the density $p$ of $\xi$ is such that 
\(
	p(y^{-})\neq p(y^{+}).
\)
Notice that the L\'evy density of the process is $s(x)=x^{-\alpha-1}+p(x)$.
Then, as shown next,
\[
    C(t):=\frac{1}{t^{1/\alpha}}\left\{\frac{1}{t}\bbp\left( X_{t}\geq{}y\right)
    -\nu([y,\infty))\right\},
\]
converges to a non-zero limit as $t\rightarrow{}0$, and so, (\ref{TargetLimit}) is infinite.
Indeed, conditioning in the number of jumps of the compound Poisson component $Y_{t}$, 
\begin{align*}
    \bbp\left( X_{t}\geq y\right)&
    =\bbp\left( S_{t}\geq y\right)
    e^{-t}+
    e^{-t}\sum_{n=1}^{\infty}\frac{t^{n}}{n!}
    \;\bbp\left( S_{t}+\sum_{i=1}^{n}\xi_{i}\geq y\right).
\end{align*}
Then,  one easily writes 
\begin{align*}
    C(t)&=t^{1-\frac{1}{\alpha}}\,
    e^{-t}\cdot \frac{1}{t}\left\{ \frac{1}{t} \bbp (S_{t}\geq{}y)-\int_{y}^{\infty} x^{-\alpha-1}dx\right\}\\
    &\quad+
    e^{-t}\, t^{-\frac{1}{\alpha}}\left\{  \bbp (S_{t}+\xi_{1}\geq{}y)-\bbp(\xi_{1}\geq{}y)\right\} +
    O(t^{1-\frac{1}{\alpha}}).
\end{align*}
Using the self-similarity of $S$, the second term on the right hand side converges to $\left(p(y^{-})- p(y^{+})\right)\bbe S_{1}^{+} $ and so, for $1<\alpha<2$, 
\[
	\lim_{t\rightarrow{}0} C(t) = \left(p(y^{-})- p(y^{+})\right)\bbe S_{1}^{+} \neq 0.
\]
\end{remk}

\section{Expansions via approximations by smooth functions}\label{ApprxmtSect}

The identity (\ref{MomentExpansionSmooth}) suggests the possibility
of achieving power expansions for
$\bbp\left(X_{t}\geq{}y\right)$ by approximating
$f(x)={\bf 1}_{\{x\geq{}y\}}$ using functions
$f_{m}$
in {\DRed $C_{b}^{2n+2}$}.
To this end, let us introduce mollifiers
$\varphi_{m}\in C^{\infty}$
with compact support contained in $[-1,1]$ that converges to the
Dirac delta function in the space of Schwartz distribution.
For concreteness,
we take $\varphi_{m}(x):=m \varphi(mx)$,
where $\varphi$ is a symmetric bump like function integrable to
$1$. Notice that
\begin{equation}\label{Mollifiers}
    f_{m}(x):=f*\varphi_{m}(x)
    =\int_{-\infty}^{x-y}\varphi_{m}(u)du,
\end{equation}
converges to $f(x)$, for any $x\neq{}y$.
Clearly, applying (\ref{MomentExpansionSmooth}) to each $f_{m}$,
\begin{equation}\label{MES2}
    \bbe f_{m}(X_{t})=
    \sum_{k=1}^{n} \frac{t^{k}}{k!}
    L^{k}f_{m}(0)+
    \frac{t^{n+1}}{n!}
    \int_{0}^{1} (1-\alpha)^{n}\bbe\left\{
    L^{n+1} f_{m}(X_{\alpha t})\right\}d\alpha,
\end{equation}
and by the dominated convergence theorem, 
\begin{equation}\label{LRH}
	\lim_{m\rightarrow{}\infty} \bbe f_{m}(X_{t})=\bbp(X_{t}\geq{}y).
\end{equation}
Thus,  the problem is to identify conditions
for the limit of each term on the right-hand side to converge as $m\rightarrow\infty$
and to identify the corresponding limiting value. 
The advantage of working with (\ref{MES2}) instead of the decomposition of $X$ of the previous section is that the coefficients $d_{k}$ of (\ref{GnrlExpansion0}) can be identified more explicitly.

As before, $c_{\varepsilon}\in C^{\infty}$ denotes a smooth truncation function such that 
\[
	{\bf 1}_{[-\varepsilon/2,\varepsilon/2]}(x)
	\leq c_{\varepsilon}(x) \leq  {\bf 1}_{[-\varepsilon,\varepsilon]}(x).
\]
The following operators will be useful in the sequel
\begin{align*}
    L_{i}g(x)&:= b_{i} g^{(i)}(x),\quad i=0,1,2\\
    L_{3}g(x)&:= \int g(x+u)\bar{c}_{\varepsilon}(u) \nu(du)\\
    L_{4}g(x)&:=\int \int_{0}^{1}
    g''(x+\beta w)(1-\beta) d\beta w^{2} c_{\varepsilon}(w) \nu(dw),
\end{align*}
where $\bar{c}_{\varepsilon}(u):= 1- c_{\varepsilon}(u)$, $b_{0}:=-\int \bar{c}_{\varepsilon}(u) \nu(du)$,
$b_{1}:=b-\int u ({\bf 1}_{\{|u|\leq{}1\}}-c_{\varepsilon}(u))
\nu(du)$ and $b_{2}:=\sigma^{2}/2$.
Note that 
\[
	Lg=\sum_{i=1}^{5}L_{i}g, 
\]
for any bounded  $g\in C^{2}_{b}$. Moreover, it turns out that the following 
commuting properties hold true 
{for any $g\in C^{2}_{b}$}:  
\[
    L_{i}L_{j}g=L_{j}L_{i}g.
\]

{
\begin{remk}\label{SmplrDcmp}
	Under additional assumptions on the L\'evy triplet  
$(\sigma^{2},b,\nu)$, we can choose more parsimonious 
decompositions of the infinitesimal generator. For instance, if one of the $b_{i}$'s is zero, 
then the corresponding operator is superfluous and can be omitted in the analysis below.  
Also, if $\int_{|w|\leq{}1} |w|\nu(dw) < +\infty$ (in which case the L\'evy process has bounded variation), 
then $L_{4}$ can be defined as: 
	\begin{align*}
		 {\DRed  L_{4}g:=\int\left(g(x+w)-g(x)\right)c_{\varepsilon}(w)\nu(dw)
	    =\int \int_{0}^{1}
    		g'(x+\beta w) d\beta w c_{\varepsilon}(w) \nu(dw)},
	\end{align*}
provided that $b_{1}$ is adjusted accordingly. If $\nu(\bbr\backslash\{0\})<\infty$, 
$L_{4}$ can be omitted, provided that we define $L_3, b_0$ and $b_1$ via: 
$L_{3}g(x)=\int g(x+u)\nu(du)$, $b_{0}=\nu(\bbr\backslash\{0\})$ and $b_{1}=b-\int_{|u|\leq{}1}u\nu(du)$.
\end{remk}
}

Let us introduce some more notation. For ${\bf k}:=(k_{0},\dots,k_{4})$
with $k_{0},\dots,k_{4}\geq{}0$,
${\bf u}:=(u_{1},\dots,u_{k_{3}})$,
${\bf w}:=(w_{1},\dots,w_{k_{4}})$, and
${\bf \beta}:=(\beta_{1},\dots,\beta_{k_{3}})$,
define the \emph{finite} measure
\begin{align*}
    {\DRed d\pi_{_{\bf k}}^{\varepsilon}({\bf u},{\bf w},{\bf \beta})=
    \prod_{i=1}^{k_{3}}\bar{c}_{\varepsilon}(u_{i})\nu(du_{i})
    \prod_{j=1}^{k_{4}}c_{\varepsilon}(w_{j})w_{j}^{2}\nu(dw_{j})
    \prod_{j=1}^{k_{4}}(1-\beta_{j})d\beta_{j}},
\end{align*}
on the space
$\widetilde{E}_{\bf k}:=\bbr^{k_{3}+k_{4}}\times [0,1]^{k_{4}}$.
Consider also the following related finite measure
\begin{align*}
    {\DRed d\tilde{\pi}^{\varepsilon}_{_{\bf k}}(u_{2},\dots,u_{k_{3}},{\bf w},{\bf \beta})=
    \prod_{i=2}^{k_{3}}\bar{c}_{\varepsilon}(u_{i})\nu(du_{i})
    \prod_{j=1}^{k_{4}}c_{\varepsilon}(w_{j})w_{j}^{2}\nu(dw_{j})
    \prod_{j=1}^{k_{4}}(1-\beta_{j})d\beta_{j}}.
\end{align*}
We sometimes drop the subscript ${\bf k}$ {\DRed and superscript $\varepsilon$} in the
measures defined above.
Also, the integral of a function
$g$ with respect to a
measure $\pi_{_{\bf k}}^{\varepsilon}$ is denoted by $\pi_{_{\bf k}}^{\varepsilon}(g)$
and we assume, by convention, that 
$\pi^{\varepsilon}_{_{\bf k}}(g)=g$, when $k_{3}=k_{4}=0$.
Similarly,
$\tilde\pi^{\varepsilon}_{_{\bf k}}(g)=g$
if $k_{4}=0$ and $k_{3}=1$ or $0$.

Let ${\bf \mathcal{K}}_{k}$ be the class of all ${\bf k}=(k_{0},\dots,k_{4})$ is such that
$k_{i}\geq{}0$ and $k_{0}+\dots+k_{4}=k$.
Note that, for any $k\geq{}1$,
\begin{equation}\label{DcmL2}
	L^{k}g (x)=\sum_{{\bf k}\in\mathcal{K}_{k}} b_{0}^{k_{0}}b_{1}^{k_{1}}b_{2}^{k_{2}}
        \binom{k}{\bf k}  B_{{\bf k}} g (x),
\end{equation} 
where $\binom{k}{\bf k}=k!/(k_{0}!\dots k_{4}!)$ is the multinomial coefficient
and
\begin{align*}
    B_{{\bf k}} g(x)
    :=
    \int g^{(k_{1}+2k_{2}+2k_{4})}
    \left(x+\displaystyle{\sum_{i=1}^{k_{3}}}u_{i}+
        \displaystyle{\sum_{j=1}^{k_{4}}}\beta_{j}w_{j}
    \right)d\pi_{_{\bf k}}.
\end{align*}

We first show that all terms in the right-hand side of (\ref{MES2}) converges. 
\begin{prop}\label{ExpnsSmoothThrm}
 Let $y>0$,  $n\geq 1$, and $0 < \varepsilon<y/(n+1)\wedge{}1$.
    Assume that $\nu$ has a  density $s$ such that  for any $0\leq k\leq{} 2n+1$ and any $\delta>0$,
    \begin{equation}\label{StndgCond0}
        a_{k,\delta}:=\sup_{|x|>\delta}|s^{(k)}(x)|<\infty.
    \end{equation}
    Then, for any $1\leq k\leq n$, 
    \begin{equation}\label{CffcntsSeries}
        \hat{d}_{k}(y):=\lim_{m\rightarrow\infty} L^{k}f_{m}(0)=
        \sum_{{\bf k}\in\mathcal{K}_{k}}
        \hat{c}_{\bf k}\binom{k}{\bf k}
        a_{\bf k},
    \end{equation}
    where, for ${\bf k}=(k_{0},\dots,k_{4})$ and $\ell_{\bf k}:=k_{1}+2k_{2}+2k_{4} $, 
    $\hat{c}_{\bf k}$ and $a_{\bf k}:=a_{\bf k}(y)$ are given by
    \begin{align}
    \hat{c}_{\bf k}&:=
        b_{0}^{k_{0}}b_{1}^{k_{1}}b_{2}^{k_{2}}
        {(-1)^{(k_{1}-1){\bf 1}\{\ell_{\bf k}>0\}}}\nonumber \\
    \label{ForCffcnts}
    a_{\bf k}&:=
    \left\{\begin{array}{ll}
    \int(\bar{c}_{\varepsilon}s)^{(\ell_{\bf k}-1)}\left(y
    -\displaystyle{\sum_{i=2}^{k_{3}}}u_{i}-
    \displaystyle{\sum_{j=1}^{k_{4}}}\beta_{j}w_{j}\right)\,
    d\tilde{\pi}_{_{\bf k}},& k_{3}>0,\,\,\ell_{\bf k}>0,\\
    \int {\bf 1}{\{\displaystyle{
    \sum_{i=1}^{k_{3}}}u_{i}
    \geq y\}}\,
    d\pi_{_{\bf k}}, & k_{3}>0, \ell_{\bf k}=0,\\
    0,& \text{otherwise}.\end{array}
    \right.
\end{align}
In particular, the limit 
   \begin{equation}\label{Rmdr}
        \mathcal{R}_{n}(t,y):=\lim_{m\rightarrow\infty}
        \int_{0}^{1}(1-\alpha)^{n}
        \bbe L^{n+1}f_{m}(X_{\alpha t})d\alpha,
    \end{equation}
    exists, and moreover, 
\begin{equation}\label{EqUs2}
    \bbp(X_{t}\geq y)=
    \sum_{k=1}^{n} \hat{d}_{k}(y)\,\frac{t^{k}}{k!}+
    \frac{t^{n+1}}{n!} \mathcal{R}_{n}(t,y).
\end{equation}
\end{prop}
\begin{proof}
We write $\hat{d}_{\bf k}$ for $\hat{d}_{\bf k}(y)$ and $\calR_{n}(t)$ for $\calR_{n}(t,y)$.
From (\ref{DcmL2}), it suffices to show that for any ${\bf k}=(k_{0},\dots,k_{4})\in\mathcal{K}_{k}$, 
\[
	\lim_{m\rightarrow\infty}B_{\bf k}f_{m}(x)  	   
    ={(-1)^{(k_{1}-1){\bf 1}\{\ell>0\}}} a_{\bf k},
\]
with $\ell:=k_{1}+2k_{2}+2k_{4}$
(recalling that by convention $\int g d\pi_{\bf k}=g$, if $k_{3}+k_{4}=0$).  
In case $\ell=0$, 
\[
	B_{\bf k}f_{m}(x) =\int f_{m}
    \left(x+\displaystyle{\sum_{i=1}^{k_{3}}}u_{i}
    \right)d\pi_{_{\bf k}}({\bf u}),
  \]
  which clearly converges to {\DRed
  \(
  	\int {\bf 1}{\{{
    	\sum_{i=1}^{k_{3}}}u_{i}\geq y\}}\,
    d\pi_{_{\bf k}},
   \)}
    since $f_{m}(x)\stackrel{m\rightarrow\infty}{\rightarrow} {\bf 1}_{x\geq{}y}$ for any $x\neq y$, 
    $|f_{m}|\leq{}1$, and $\pi_{\bf k}$ is a non-atomic finite measure. 
    
    Consider the case $k_{3},\ell>0$. 
    Writing $z=\sum_{i=2}^{k_{3}}u_{i}+
    \sum_{j=1}^{k_{4}}\beta_{j}w_{j}$ and integrating by parts,
\begin{align*}
    &\int f_{m}^{(\ell)}
    \left(x+u_{1}+z\right)
    \bar{c}_{\varepsilon}\cdot s(u_{1})du_{1}=\int
    \varphi_{m}^{(\ell-1)}
    \left(x+u_{1}+z-y\right)
    \bar{c}_{\varepsilon}\cdot s(u_{1})du_{1}\\
    &\quad=(-1)^{k_{1}-1}\int
    \varphi_{m}
    \left(x+u_{1}+z-y\right)
    (\bar{c}_{\varepsilon}s)^{(\ell-1)}(u_{1}) du_{1}\\
    &\quad\stackrel{m\rightarrow\infty}{\longrightarrow}
    (-1)^{k_{1}-1}
    (\bar{c}_{\varepsilon}s)^{(\ell-1)}(y-x-z),
\end{align*}
provided that $c\in C^{\ell-1}(\bbr\backslash\{0\})$.
Moreover,  we have that 
\[
    \left|\int f_{m}^{(\ell)}
    \left(x+u_{1}+z\right)
    \bar{c}\cdot s(u_{1})du_{1}\right|\leq
    2^{\ell-1} \max_{k\leq{}\ell-1}
    \sup_{|x|>\delta}
    |s^{(k)}(x)|<\infty.
\]
Thus, applying first Fubini's theorem and then 
the dominated convergence theorem give: 
\begin{align*}
 &\lim_{m\rightarrow\infty}
    B_{{\bf k}}f_{m}(x)
    = 
    (-1)^{k_{1}-1}\int
 (\bar{c}_{\varepsilon}s)^{(\ell-1)}(y- x
    -\sum_{i=2}^{k_{3}}u_{i}-
    \sum_{j=1}^{k_{4}}\beta_{j}w_{j})
    d\tilde{\pi}
\end{align*}
In case $k_{3}=0$ and $\ell>0$, 
\[
    B_{\bf k}f_{m}(0)=
    \int f_{m}^{\ell}\left(\sum_{j=1}^{k_{4}}
    \beta_{j}w_{j}\right) d\pi_{_{\bf k}}
    =\int \varphi_{m}^{\ell-1}\left(\sum_{j=1}^{k_{4}}
    \beta_{j}w_{j}-y\right) d\pi_{_{\bf k}}=0,
\]
for $m$ large enough since, by construction, $\varepsilon$ is
chosen so that $y-(n+1)\varepsilon>0$, and
$\beta_{j}w_{j}$ takes values in $[-\varepsilon,\varepsilon]$ 
on the support of $\pi_{_{\bf k}}$.  Then, the existence of (\ref{Rmdr}) and the identity 
(\ref{EqUs2}) follow from 
(\ref{MES2}) and (\ref{LRH}).
\end{proof}

Notice that we cannot yet conclude that $\hat{d}_{\bf k}:=\hat{d}_{\bf k}(y)$ are the same constants as the $d_{k}$s in equations (\ref{GnrlExpansion1}) and (\ref{Cnst1}),
since we have not shown that 
  \begin{equation}\label{Lmt}
  	\limsup_{t\rightarrow{}0}| \calR_{n}(t,y)|<\infty.
\end{equation}
Actually, in view of Theorem \ref{MainExpansion}, these two conditions are equivalent; namely, 
$d_{k}=\hat{d}_{k}$ for all $k\leq{}n$ if and only if (\ref{Lmt}) holds.
We show these facts in the following result.
\begin{thrm}
	Under the conditions of Theorem \ref{MainExpansion}, (\ref{Lmt}) holds and moreover, 
	$\hat{d}_{k}=d_{k}$ in (\ref{CffcntsSeries}), which are independent of $\varepsilon$ and given by
    \[
        \lim_{t\rightarrow{}0}
        \frac{1}{t^{n}}\left\{
        \bbp\left(X_{t}\geq{}y\right)-
        \sum_{k=1}^{n-1} \hat{d}_{k}\,\frac{t^{k}}{k!}
        \right\}=\frac{\hat{d}_{n}}{n!},
    \]	
    for $k\leq{}n$.
\end{thrm}
\begin{proof}
	Using a proof as in Theorem  \ref{MainExpansion}, we conclude that 
	\begin{align}\label{FExp2}
    \bbe f_{m}(X_{t})
    =e^{-\lambda_{\varepsilon}t}\sum_{j=1}^{n}c_{j,m}\, \frac{t^{j}}{j!}
    +O_{\varepsilon}(t^{n+1}),
\end{align}
where 
{\DRed
\(
	c_{j,m}:=   
    \sum_{i=1}^{j} \binom{j}{i} L_{\varepsilon}^{j-i}\hat{f}_{i,m}(0),
\)}
with $\hat{f}_{i,m}(x):=\int f_{m}(x+u) \bar{s}_{\varepsilon}^{*i}(u) du$.  As in Remark \ref{Simplify} (i), 
(\ref{FExp2}) leads to 
\begin{equation}\label{GnrlExpansion2}
       \bbe f_{m}(X_{t})=
        \sum_{k=1}^{n} d_{k,m}\,\frac{t^{k}}{k!}+ O_{\varepsilon}(t^{n+1}),
    \end{equation}
with 
\begin{equation}\label{Cnst2}
	d_{k,m}=\sum_{j=1}^{k}  \binom{k}{j} c_{j,m} (-\lambda_{\varepsilon})^{k-j}.
\end{equation}
Since the last term in (\ref{MES2}) is $O_{\varepsilon}(t^{n+1})$, we have 
\[
	d_{k,m}=L^{k}f_{m}(0).
\]
To show that $\hat{d}_{k}:=\lim_{m\rightarrow\infty} d_{k,m}$ is identical to $d_{k}$ 
of (\ref{GnrlExpansion1})-(\ref{Cnst1}), it suffices that $\lim_{m\rightarrow\infty}c_{j,m}=c_{j}$, or equivalently, that
\[
	\lim_{m\rightarrow\infty} L_{\varepsilon}^{k}\hat{f}_{i,m}(0)= L_{\varepsilon}^{k}\hat{f}_{i}(0),
\]
for all $k\geq{}0$ and $i\geq{}1$.
The case $k=0$ is clear. For $k\geq{}1$, from (\ref{DcmL}), we only need to have
\begin{align}\label{AuxE}
   \lim_{m\rightarrow\infty}\int \hat{f}_{i,m}^{(p)}
    \left(x+
        \displaystyle{\sum_{\ell=1}^{k}}\beta_{\ell}w_{\ell}
    \right)d\pi_{k}^{\varepsilon}=\int \hat{f}_{i}^{(p)}
    \left(x+
        \displaystyle{\sum_{\ell=1}^{k}}\beta_{\ell}w_{\ell}
    \right)d\pi_{k}^{\varepsilon},
\end{align}
for any $k\geq{}0$ and $p\geq{}1$. Since $\varphi_{m}$ is symmetric,
\[
	 \hat{f}_{i,m}(x)=\int_{y-x}^{\infty} \varphi_{m}* \bar{s}_{\varepsilon}^{*i} (u) du,
\]
and thus, $\hat{f}_{i,m}'(x)=\varphi_{m}* \bar{s}_{\varepsilon}^{*i} (y-x)$, which converges to 
$\hat{f}'_{i}(x)= \bar{s}_{\varepsilon}^{*i} (y-x)$, as $m\rightarrow\infty$, uniformly in $x$. Then, 
\[
	\hat{f}_{i,m}^{(p)}(x)=(-1)^{p-1}\varphi_{m}* \bar{s}_{\varepsilon}^{*(i-1)} * \bar{s}_{\varepsilon}^{(p-1)}(y-x),
\]
which converges  to 
\[
	(-1)^{p-1}\bar{s}_{\varepsilon}^{*(i-1)} * \bar{s}_{\varepsilon}^{(p-1)}(y-x) =\hat{f}_{i}^{(p)}(x),
\]
as $m\rightarrow\infty$, uniformly in $x$. Since $\hat{\pi}_{k}^{\varepsilon}$ is a finite measure, 
(\ref{AuxE}) holds. {\DRed We have just} proved that $\hat{d}_{k}=d_{k}$, for $k\leq n$, 
and by matching (\ref{GnrlExpansion1}) and (\ref{EqUs2}), it follows that
\(
	\calR_{n}(t) = O(1), 
\)
as $t\rightarrow{}0$. The last two statements of the result are easily proved by induction.
\end{proof}

\section{The remainder and expansions for the transition densities}\label{DnstyExp}
In this section we give a more explicit expression for the remainder $\calR_{n}$ in (\ref{Rmdr}), whose existence was proved in Theorem  \ref{ExpnsSmoothThrm}. In order to do this, we expand 
$L^{n+1}f_{m}(X_{\alpha t})$ using (\ref{DcmL2}) and show that the limit of the resulting terms exists. The hardest case to tackle corresponds to $k_{3}=0$ and $\ell>0$, where we will need to impose the following condition on the transition density $p_{t}$ of $X_{t}$:
\begin{equation}\label{StndgCond}
	 c_{k,\delta}:=\sup_{0<u\leq{}t_{0}}
        \sup_{|x|>\delta}|p_{u}^{(k)}(x)|<\infty,
\end{equation}
for any $\delta>0$ and for some $t_{0}>0$. 
Condition (\ref{StndgCond})  is reasonable 
since it is known that 
\begin{equation}\label{LmtDnstyUnif}
	\lim_{t\rightarrow{}0} \sup_{|x|>\varepsilon}\left|\frac{1}{t} p_{t}(x)-s(x)\right|=0,
\end{equation}
(see Proposition III.6 in \cite{Leandre}  and Corollary 1.1 in \cite{Ruschendorf}). We confess however that, in general, 
this condition might be hard to verify since
the transition densities $p_{t}$ of a L\'evy model 
are not explicitly given  in many cases.
Let us point out that, under certain conditions,  Picard \cite{Picard:1997} proves that 
\begin{equation}\label{EqUnfrmBndDrvt}
	\sup_{x}|p_{t}^{(k)}(x)|\leq{} t^{-(k+1)/\beta},
\end{equation}
where $\beta$ is the Blumenthal-Getoor index
of $X$.  The approach in \cite{Picard:1997} was built on earlier methods and results of L\'eandre \cite{Leandre}, who proves {(\ref{LmtDnstyUnif}) and (\ref{EqUnfrmBndDrvt}) for $k=0$} using Malliavin calculus. In view of (\ref{EqUnfrmBndDrvt}),  for values of $t$ away from $0$, the derivatives of $p_{t}$ are uniformly bounded, and condition (\ref{StndgCond})  is then related to the behavior of $p_{t}^{(k)}$ when $t\rightarrow{}0$.
In Sections \ref{SectSymmtLevy} and \ref{SectGnrlSymmtLevy}, we prove that the 
condition (\ref{StndgCond}) holds for symmetric stable 
L\'evy processes and other related processes, rising the hope to use similar methods in other cases.
As in the previous section, we take $y>0$, $n\geq{}1$, and $\varepsilon>0$ such that
	\begin{equation}\label{CndSpprt}
		0 < \varepsilon<y/(n+1)\wedge{}1.
	\end{equation}
\begin{thrm}\label{RmdrConv}
 Assume that $\nu$ has a  density $s$ such that  
 (\ref{StndgCond0}) holds for any $0\leq k\leq{} 2n+1$ and any $\delta>0$. Also, assume that there exists a $t_{0}>0$ such that for all $0<t<t_{0}$, 
 $X_{t}$ has a $C^{2n+1}$ density $p_{t}$ satisfying (\ref{StndgCond}) for any $0\leq k\leq{} 2n+1$ and any $\delta>0$. Then, the remainder 
    \[
        \mathcal{R}_{n}(t,y):=\lim_{m\rightarrow\infty}
        \int_{0}^{1}(1-\alpha)^{n}
        \bbe L^{n+1}f_{m}(X_{\alpha t})d\alpha,
    \]
    is given by
    \begin{equation}\label{Rmndr}
        \mathcal{R}_{n}(t,y)=
        \sum_{{\bf k}\in{\bf\mathcal{K}}_{n+1}}
        c_{\bf k}\binom{n+1}{\bf k}
        \int_{0}^{1}(1-\alpha)^{n}
        a_{\bf k}(t;\alpha,y)d\alpha,
    \end{equation}
    where, for ${\bf k}=(k_{0},\dots,k_{4})\in{\bf\mathcal{K}}_{n+1}$,
    $c_{\bf k}$, and $a_{\bf k}(t;\alpha)$ are defined via:
    \begin{align*}
        c_{\bf k}&:=
        b_{0}^{k_{0}}b_{1}^{k_{1}}b_{2}^{k_{2}}
        {(-1)^{(k_{1}-1){\bf 1}\{\ell>0\}}},\\
        a_{\bf k}(t;\alpha,y)&:=\left\{\begin{array}{ll}
        \int \bbp\left(X_{\alpha t}+
        \displaystyle{\sum_{i=1}^{k_{3}}}u_{i}
        \geq{}y \right)
        d \pi_{_{\bf k}},
        & \ell=0 \\
        \int \bbe (\bar{c}s)^{(\ell-1)}\left(y- X_{\alpha t}
        -\displaystyle{\sum_{i=2}^{k_{3}}}u_{i}-
        \displaystyle{\sum_{j=1}^{k_{4}}}\beta_{j}w_{j}
        \right) d\tilde{\pi}_{_{\bf k}},
        &k_{3}>0, \ell>0\\
        \int p_{\alpha t}^{(\ell-1)}\left(y-
        \displaystyle{\sum_{j=1}^{k_{4}}}\beta_{j}w_{j}\right)
        d\pi_{_{\bf k}},
        &k_{3}=0, \ell>0,
        \end{array}\right.
    \end{align*}
    with $\ell:=k_{1}+2k_{2}+2k_{4}$.  
\end{thrm}
\begin{proof}
From (\ref{DcmL2}), it suffices to show that for any ${\bf k}=(k_{0},\dots,k_{4})\in\mathcal{K}_{k}$, 
\[
	\lim_{m\rightarrow{}\infty}
	\int_{0}^{1}(1-\alpha)^{n} \bbe 
    	B_{{\bf k}} f_{m}(X_{\alpha t}) d\alpha=
	{(-1)^{(k_{1}-1){\bf 1}\{\ell>0\}}} 
	\int_{0}^{1} (1-\alpha)^{n} a_{\bf k}(t;\alpha)d\alpha.
\]
Below, $T_{\bf k}g(x;\cdot)$ is the function defined on
$E_{\bf k}:=\bbr^{k_{3}+k_{4}}\times [0,1]^{k_{4}}$ via 
\begin{align*}
   & T_{{\bf k}}g(x;u_{1},\dots,u_{k_{3}},
    w_{1},\dots,w_{k_{4}},
    \beta_{1},\dots,\beta_{k_{4}})\\
    \quad&=g^{(k_{1}+2k_{2}+2k_{4})}
    \left(x+\sum_{i=1}^{k_{3}}u_{i}+
    \sum_{j=1}^{k_{4}}\beta_{j}w_{j}\right).
\end{align*}

We break our proof in different cases. 
Suppose first that $\ell:=k_{1}+2k_{2}+2k_{4}=0$.
Since $0\leq f_{m}\leq 1$, apply Fubini's theorem to get:
\[
    \bbe B_{{\bf k}}f_{m}(X_{\alpha t})
    = b_{0}^{k_{0}}b_{1}^{k_{1}}b_{2}^{k_{2}}
    \pi_{_{\bf k}}(\bbe T_{\bf k}f_{m}(X_{\alpha t},\cdot)).
\]
Since $X_{\alpha t}$ is a continuous random variable,
the dominated convergence theorem implies that
\begin{align*}
    \bbe f_{m}\left(X_{\alpha t}+\sum_{i=1}^{k_{3}}u_{i}\right)
    \,\stackrel{m\rightarrow\infty}{\longrightarrow}\,
    \bbp\left(X_{\alpha t}\geq{}y
    -\sum_{i=1}^{k_{3}}u_{i}\right).
\end{align*}
Again, by dominated convergence, 
{\DRed $\int_{0}^{1}(1-\alpha)^{n}\bbe B_{{\bf k}}f_{m}(X_{\alpha t})d\alpha$ converges to $\int_{0}^{1}(1-\alpha)^{n}
 \int \bbp\left(X_{\alpha t}\geq{}y
    -\sum_{i=1}^{k_{3}}u_{i}\right)
    d\pi
    \,d\alpha$}.
    
Next, we consider the case $\ell>0$ and $k_{3}=0$.
Again by Fubini's theorem,
\[
    \bbe B_{{\bf k}}f_{m}(X_{\alpha t})
    = 
    \pi(\bbe T_{\bf k}f_{m}(X_{\alpha t},\cdot)).
\]
Writing $z=
    \sum_{j=1}^{k_{4}}\beta_{j}w_{j}$,
integrating by parts, and changing variables,
we have
\begin{align*}
    \bbe f_{m}^{(\ell)}
    \left(X_{\alpha t}+z\right)
    &=\int \varphi_{m}^{(\ell-1)}
    \left(x+z-y\right) p_{\alpha t}(x) dx\\
    &=(-1)^{k_{1}-1}\int
    \varphi_{m}
    \left(x\right) p_{\alpha t}^{(\ell-1)}(x+y-z) dx,
\end{align*}
{\DRed which converges to $(-1)^{k_{1}-1} p_{\alpha t}^{(\ell-1)}(y-z)$ as 
$m\rightarrow\infty$},
if $p_{\alpha t}^{(\ell-1)}$ is continuous.
Moreover, under (\ref{StndgCond}) and
with the help of (\ref{CndSpprt}), for $m$ large enough,
\[
    \sup_{0<\alpha\leq 1}
    \sup_{\beta_{j},w_{j}}\left|\bbe f_{m}^{(\ell)}
    \left(X_{\alpha t}+\sum_{j}\beta_{j}w_{j}\right)\right|\leq
    \sup_{0<\alpha<1}\sup_{|x|>\delta}
    |p_{\alpha t}^{(\ell-1)}(x)|<\infty,
\]
taking $\delta:=(y-(n+1)\varepsilon)/2$.
Then, by dominated convergence: 
\begin{align*}
 &\lim_{m\rightarrow\infty}
    \int_{0}^{1}(1-\alpha)^{n}\bbe B_{{\bf k}}f_{m}(X_{\alpha t})d\alpha
    = \\
 &\quad(-1)^{k_{1}-1}
\int_{0}^{1}(1-\alpha)^{n}\int p_{\alpha t}^{(\ell-1)}(y-
    \sum_{j=1}^{k_{4}}\beta_{j}w_{j})
    d\pi
    (w_{1},\dots,w_{k_{4}},\beta_{1},\dots,\beta_{k_{4}})
    \,d\alpha.
\end{align*}
Note that the previous limiting value is uniformly bounded in $t$ and $y$
by
\[
    \frac{1}{n+1}\pi(\bbr^{2k_{4}})
    \sup_{0<u\leq{}t}\sup_{|x|>\delta}
    |p_{u}^{(\ell-1)}(x)|<\infty.
\]
The only remaining case to tackle is
when $\ell>0$ and $k_{3}>0$.
Writing $z=\sum_{i=2}^{k_{3}}u_{i}+
    \sum_{j=1}^{k_{4}}\beta_{j}w_{j}$, we have that
\begin{align*}
    &\int f_{m}^{(\ell)}
    \left(X_{\alpha t}+u_{1}+z\right)
    (\bar{c}s)(u_{1})du_{1}=\int
    \varphi_{m}^{(\ell-1)}
    \left(X_{\alpha t}+u_{1}+z-y\right)
    (\bar{c}s)(u_{1})du_{1}\\
    &\quad=(-1)^{k_{1}-1}\int
    \varphi_{m}
    \left(X_{\alpha t}+u_{1}+z-y\right)
    (\bar{c}s)^{(\ell-1)}(u_{1}) du_{1},
\end{align*}
{\DRed which converges to $(-1)^{k_{1}-1}
    (\bar{c}s)^{(\ell-1)}(y-X_{\alpha t}-z)$ as $m\rightarrow\infty$,}
provided that $c\in C^{\ell-1}(\bbr\backslash\{0\})$.
Moreover, under (\ref{StndgCond0}),
we have that 
\[
    \left|\int f_{m}^{(\ell)}
    \left(X_{\alpha t}+u_{1}+z\right)
    (\bar{c}s)(u_{1})du_{1}\right|\leq
    2^{\ell-1} \max_{k\leq{}\ell-1}
    \sup_{|x|>\varepsilon}
    |s^{(k)}(x)|<\infty.
\]
Thus, applying first Fubini's theorem and then 
the dominated convergence theorem give: 
\begin{align*}
 &\lim_{m\rightarrow\infty}
    \int_{0}^{1}(1-\alpha)^{n}\bbe B_{{\bf k}}f_{m}(X_{\alpha t})d\alpha
    = \\
 &
    (-1)^{k_{1}-1}\int (1-\alpha)^{n}\int
 \bbe (\bar{c}s)^{(\ell-1)}(y- X_{\alpha t}
    -\sum_{i=2}^{k_{3}}u_{i}-
    \sum_{j=1}^{k_{4}}\beta_{j}w_{j})
    d\tilde{\pi}\, d\alpha.
\end{align*}
This last case achieves the proof of the theorem. \end{proof}

\begin{remk}
	From the proof is clear that 
    \[
        \left|\mathcal{R}_{n}(t,y)\right|\leq
        \frac{\alpha_{n}}{n+1}
        \,\max\left\{\max_{k\leq 2n+1}c_{k,\varepsilon},
        \max_{k\leq{}2n+1}a_{k,\varepsilon},1\right\},
    \]
    for any $t\in (0,t_{0})$,
    where
    \[
        \alpha_{n}:=\sum_{{\bf k}\in{\bf\mathcal{K}}_{n+1}}
        |c_{\bf k}|\binom{n+1}{\bf k}
        \max\{\pi_{_{\bf k}}(1),\tilde\pi_{_{\bf k}}(1),
        \hat\pi_{_{\bf k}}(1)\}.
    \]
    This bound on $\calR(t,y)$ is valid for any $y\geq \underline{y}$, taking also $\varepsilon$ such that  $0<\varepsilon<\underline{y}/(n+1)\wedge 1$.
\end{remk}

One of the advantages of an explicit expression for the remainder $\calR_{n}(t,y)$ of (\ref{EqUs2}) is that we can obtain small time expansions in $t$ for the transition density $p_{t}(y)$ of $X_{t}$, 
by a formal differentiation of (\ref{EqUs2}). With this application in mind we need to show that the coefficients $a_{\bf k}(y):=a_{\bf k}$ of (\ref{ForCffcnts}) and 
\[
	A_{\bf k}(t,y):=\int_{0}^{1}(1-\alpha)^{n}a_{\bf k}(t;\alpha,y)d\alpha,
\]
of (\ref{Rmndr}) are differentiable in $y$. The following result corrects Theorem 1 from \cite{Ruschendorf}.
\begin{prop}\label{ExpnsDnstyThrm}
 Let $y>0$,  $n\geq 1$, and $0 < \varepsilon<y/(n+1)\wedge{}1$.
 Assume that $\nu$ has a  density $s$ such that  
 (\ref{StndgCond0}) holds for any $0\leq k\leq{} 2n+2$ and any $\delta>0$. Also, assume that there exists a $t_{0}>0$ such that for all $0<t<t_{0}$, 
 $X_{t}$ has a $C^{2n+2}$ density $p_{t}$ satisfying (\ref{StndgCond})  for any $0\leq k\leq{} 2n+2$ and any $\delta>0$. Then, 
 \begin{equation}\label{EqUs2b}
    p_{t}(y)=
    \sum_{k=1}^{n} \hat{d}_{k}'(y)\,\frac{t^{k}}{k!}+ O_{\varepsilon}(t^{n+1}),
\end{equation}
where for any $1\leq k\leq n$, 
    \begin{equation}\label{CffcntsSeriesb}
        \hat{d}'_{k}(y):=
        \sum_{{\bf k}\in\mathcal{K}_{k}}
        \hat{c}_{\bf k}\binom{k}{\bf k}
        a'_{\bf k}(y),
    \end{equation}
    and, for ${\bf k}=(k_{0},\dots,k_{4})$ and $\ell_{\bf k}:=k_{1}+2k_{2}+2k_{4}$,
    \begin{align}
    \hat{c}_{\bf k}&:=
        b_{0}^{k_{0}}b_{1}^{k_{1}}b_{2}^{k_{2}}
        {(-1)^{(k_{1}-1){\bf 1}\{\ell_{\bf k}>0\}}}\nonumber \\
    \label{ForCffcntsb}
    a'_{\bf k}(y)&:=
    \left\{\begin{array}{ll}
    \int(\bar{c}_{\varepsilon}s)^{(\ell_{\bf k})}\left(y
    -\displaystyle{\sum_{i=2}^{k_{3}}}u_{i}-
    \displaystyle{\sum_{j=1}^{k_{4}}}\beta_{j}w_{j}\right)\,
    d\tilde{\pi}_{_{\bf k}},& k_{3}>0,\,\,\ell_{\bf k}>0,\\
    \int (\bar{c}_{\varepsilon}s)\left(y-\displaystyle{\sum_{i=2}^{k_{3}}}u_{i}\right)\,
    d\tilde{\pi}_{_{\bf k}}, & k_{3}>0, \ell_{\bf k}=0,\\
    0,& \text{otherwise}.\end{array}
    \right.
\end{align}
\end{prop}
\begin{proof}
	In view of (\ref{EqUs2}) and (\ref{Rmndr}), we {\DRed first} have to {\DRed show} that the derivative of (\ref{ForCffcnts}), for each case, exists and is given by (\ref{ForCffcntsb}). This will follow from the fact that $\tilde{\pi}_{\bf k}$ and $\pi_{\bf k}$ are finite measures and the integrands have continuous uniformly bounded derivatives with respect to $y$. 
Also, we need to show that 
the derivatives, with respect to $y$, of $a_{\bf k}(t;\alpha,y)$ exist and are continuous as well as  bounded {\DRed for $t$ small}. When $\ell=0$ and $k_{3}=0$, 
\[
	A_{\bf k}(t,y)=\int_{0}^{1} (1-\alpha)^{n} a_{\bf k}(t;\alpha,y)=  
	\int_{0}^{1} (1-\alpha)^{n}\int_{y}^{\infty}p_{\alpha t} (z) dz d\alpha.
\]
From (\ref{LmtDnstyUnif}), there exist $K>0$ and $t_{0}>0$ such that 
$\sup_{0<u<t_{0}}\sup_{|x|>\delta} p_{u}(x)<K$. Hence, one can interchange derivation and integration:
\begin{equation}\label{AE3}
	\frac{\partial A_{\bf k}(t,y)}{\partial y}=\int_{0}^{1} (1-\alpha)^{n}p_{\alpha t}(y)d\alpha,
\end{equation}
and moreover, the supremum of (\ref{AE3}) over $0<t<t_{0}$ is finite.
In case of $\ell=0$ and $k_{3}>0$, one can write
\[
	A_{\bf k}(t,y)=\int_{0}^{1} (1-\alpha)^{n} \int_{y}^{\infty}
	\int p_{\alpha t}\left(z-
        \displaystyle{\sum_{i=1}^{k_{3}}}u_{i}\right) d\pi_{\bf k}d z d\alpha.
\]
The inner integral is continuous in $z$ and is such that 
\[
	0\leq \bar{p}_{\alpha t}(z):=
	\int\dots \int  p_{\alpha t}\left(z-
        \displaystyle{\sum_{i=1}^{k_{3}}}u_{i}\right) (\bar{c}_{\varepsilon} s)(u_{1})du_{1}d\tilde\pi_{\bf k}\leq \sup_{u_{1}}\bar{s}_{\varepsilon}(u_{1}) \lambda_{\varepsilon}^{k_{3}-1}.
\]
Thus,  one can interchange derivation and integration to get:
\[
	\frac{\partial A_{\bf k}(t,y)}{\partial y}=\int_{0}^{1} (1-\alpha)^{n}\bar{p}_{\alpha t}(y)d\alpha,
\]
and moreover, the supremum over $t>0$ is finite. 
In  case $k_{3}>0$ and $\ell>0$, 
\[
	A_{\bf k}(t,y)= \int_{0}^{1}(1-\alpha)^{n}\int \bbe (\bar{c}s)^{(\ell-1)}\left(y- X_{\alpha t}
        -\displaystyle{\sum_{i=2}^{k_{3}}}u_{i}-
        \displaystyle{\sum_{j=1}^{k_{4}}}\beta_{j}w_{j}
        \right) d\tilde{\pi}_{_{\bf k}} d\alpha.
\]
Assuming that $\bar{c} s\in C^{\ell}_{b}$, one can interchange the derivative with respect to $y$ and the integral, and the resulting term will be uniformly bounded in $t$ since $\widetilde{\pi}_{\bf k}$ is a finite measure. 
Finally, when $k_{3}=0$ and $\ell>0$, 
\[
	A_{\bf k}(t,y)= \int_{0}^{1}(1-\alpha)^{n}\int p_{\alpha t}^{(\ell-1)}\left(y-
        \displaystyle{\sum_{j=1}^{k_{4}}}\beta_{j}w_{j}\right)
        d\pi_{_{\bf k}}d\alpha.
\]
Assuming that $p_{\alpha t}\in C^{\ell}_{b}$ and satisfies (\ref{StndgCond}) with $k=\ell$, one can interchange the derivative with respect to $y$ and the integral, and the resulting term will be uniformly bounded in $0<t<t_{0}$ since ${\pi}_{\bf k}$ is a finite measure. 
All previous cases will imply that 
\[	
	\frac{\partial A_{\bf k}(t,y)}{\partial y}
\]
exists and is uniformly bounded in $0<t<t_{0}$. Hence, the derivative of the last term in (\ref{EqUs2}) is $O_{\varepsilon}(t^{n+1})$. 
\end{proof}
}

\section{Symmetric stable L\'evy processes}\label{SectSymmtLevy}

In this section, we analyze the assumption
(\ref{StndgCond}), needed for the validity of Theorem \ref{RmdrConv} and \ref{ExpnsDnstyThrm},
in the case of symmetric stable  L\'evy processes. 

Let us assume that the L\'evy triplet $(\sigma^{2},b,\nu)$ 
is such that $b=0$ and that $\nu$ symmetric.  
Furthermore, let us assume that
    \begin{equation}\label{CndtnDnsty}
        \liminf_{\varepsilon\rightarrow{}0}
        \frac{\int_{[-\varepsilon,\varepsilon]} x^{2} \nu(dx)}
        {\varepsilon^{2-\alpha}}
        >0,
    \end{equation}
for $0<\alpha<2$.
{
Condition (\ref{CndtnDnsty}) is equivalent to
\[
	\lim_{\varepsilon\rightarrow0}
	\varepsilon^{\alpha}
	\int_{\{|x|>\varepsilon\}} \nu(dx)>0.
\]
}
Condition (\ref{CndtnDnsty}) is known to be sufficient for $X_{t}$ to have
a $C^{\infty}$-density $p_{t}$
(see e.g. \cite[Theorem I.1]{Leandre} or \cite[Proposition 28.3]{Sato}).  It will be useful to outline
the proof of this result.
The first step is to bound the characteristic function
$\psi_{t}(u)=\bbe e^{iuX_{t}}$
as follows:
\begin{equation}\label{BndChrcFnct}
    \left|\psi_{t}(u)\right|\leq e^{-c t |u|^{\alpha}},
\end{equation}
which is valid for $u$ large enough 
(cf. page 190 in \cite{Sato}).
Note that the right hand side of (\ref{BndChrcFnct}) is the characteristic 
function of a symmetric $\alpha$-stable L\'evy process. 
In particular,
\[
    \int \left|\psi_{t}(u)\right| |u|^{n} du<\infty,
\]
for any $n=0,\dots$, and 
the following inversion formula for $p_{t}^{(n)}$
holds:
\begin{equation}\label{GnrlInvForm}
    p_{t}^{(n)}(x)=\frac{(-i)^{n}}{2\pi}
    \int e^{-iux}u^{n}\psi_{t}(u) du,
\end{equation}
see \cite[Proposition 2.5]{Sato}.  Finally, the
Riemann-Lebesgue lemma implies that
\[
    \lim_{|x|\rightarrow\infty}p_{t}^{(n)}(x)=0.
\]

Let us try to modify the above argument for our purposes. 
In the case that $b=0$ and $\nu$ is symmetric, 
$\psi_{t}(u)$ is positive real and even, and thus,
{
\begin{equation}\label{SymmtInvForm}
    p_{t}^{(n)}(x)=
    \left\{
    \begin{array}{ll}
    \frac{(-1)^{n/2}}{\pi}\displaystyle{\int_{0}^{\infty}} \cos(ux)u^{n}\psi_{t}(u) du,& \text{if n is even,}\\
    \\
    \frac{(-1)^{(n+1)/2}}{\pi}\displaystyle{\int_{0}^{\infty}} \sin(ux)u^{n}\psi_{t}(u) du,& \text{if n is odd.}
    \end{array}
    \right.
\end{equation}
}
In light of (\ref{BndChrcFnct}), it is important to analyze
the case of a symmetric $\alpha-$stable L\'evy process. It is not surprising that
a great deal is known for this class (see e.g.
Section 14 in \cite{Sato}). For instance,
from the self-similarity property $X_{t}\ed t^{1/\alpha} X_{1}$,
\[
    p_{t}(x)= t^{-1/\alpha} p_{1}\left(t^{-1/\alpha}x\right).
\]
Asymptotic power series in  $x$ are available for $p_{1}(x)$,
from which one can also obtain the following
asymptotic behavior of
$p_{1}(x)$ when $x\rightarrow\infty$:
\begin{equation}\label{AsymptStable}
    p_{1}(x)\sim
    x^{-\alpha-1}.
\end{equation}
Note that (\ref{AsymptStable})
is consistent with the well-known asymptotic
result that
\[
    \lim_{t\rightarrow{}0}\frac{1}{t}\, p_{t}(x)=s(x)=x^{-\alpha-1},
\]
for any $x\neq 0$ (see e.g. \cite[Corollary 1]{Ruschendorf}).

We want to show that the condition (\ref{StndgCond}) holds for
symmetric stable distributions (and possibly for more general
symmetric distributions satisfying (\ref{CndtnDnsty})).
With this goal in mind, we give a method
to bound $x^{\alpha+1}p_{1}(x)$. First, we need the following
lemma:
{
\begin{lmma}\label{lmEsyEst}
    Let $\phi:(0,\infty)\rightarrow\bbr_{+}$
    be an integrable function. Then, the following statements hold:
    \begin{enumerate}
    	\item[(i)] If $\phi$ is monotone decreasing and 
	there exists $\beta\in[0,1]$ such that
    	\begin{equation}\label{RglrVar0Cond}
    		\limsup_{u\downarrow{}0} \phi(u)u^{\beta}<\infty, 
	\end{equation}
    then there exists a constant $c<\infty$, independent of $x$, such that 
    \begin{equation}\label{EqEsyEst1&2}
    		\left|\int_{0}^{\infty} \kappa(ux) \phi(u) du\right| \leq
		\frac{c}{x^{1-\beta}},
    \end{equation}
    where $\kappa$ can be either $\cos$ or $\sin$.
     
    \item[(ii)] If $\phi$ is unimodal with mode $u^{*}$, then
    \begin{equation}\label{EqEsyEst3}
    		\left|\int_{0}^{\infty} \kappa(ux) \phi(u) du\right| \leq
		\frac{\pi\phi(u^{*})}{x},
     \end{equation}
     for all $x>0$,
     where $\kappa$ can be either the function $\cos$ or $\sin$.
     Moreover, if $\phi$ is continuous, then  
     \begin{equation}\label{Limit}
    		\lim_{x\rightarrow\infty}
		x\int_{0}^{\infty} \kappa(ux) \phi(u) du= \pi \phi(u^{*}).
     \end{equation}
     \end{enumerate}
\end{lmma}
\begin{proof}
To show (i), we first note the following two easy inequalities:
	\begin{align}\label{EqEsyEst1}
       		 0\leq \int_{0}^{\infty} \sin (ux)\phi(u)du&\leq{}
       	 	\int_{0}^{\frac{\pi}{x}}\phi(u)du<\infty,\\
        	\label{EqEsyEst2}
        	\left|\int_{0}^{\infty} \cos (ux)\phi(u)du\right|
        	&\leq{} \int_{0}^{\frac{3\pi}{2x}}\phi(u)du<\infty,
    \end{align} 
valid for any nonnegative function $\phi$ that is decreasing and integrable. Therefore, 
if $\kappa$ is either $\cos$ or $\sin$, then
\begin{equation*}
        \left|\int_{0}^{\infty} \kappa (ux)\phi(u)du\right|
        \leq{} \int_{0}^{\frac{2\pi}{x}}\phi(u)du.
\end{equation*} 
In view of the condition (\ref{RglrVar0Cond}), there exists a $c'>0$ and $x_{0}>0$ such that 
for all $x>x_{0}$, $\phi(u)\leq c' u^{-\beta}$, in $(0,2\pi/x]$, for any $x>x_{0}$. Then, 
(\ref{EqEsyEst1&2}) is clear for $x>x_{0}$.  The values 
$x\leq x_{0}$ can be taken care of easily since 
\[
	\int_{0}^{2\pi/x}\phi(u)du\uparrow \int_{0}^{\infty}\phi(u)du,
\]
when $x\searrow{}0$. 
    Let us now show (ii).  First, set 
    \[
    	q(x):=\int_{0}^{\infty}\kappa(ux)\phi(u)du.
	\] 
    By assumption,  $\phi$ 
    is increasing on $[0,u^{*}]$, and decreasing on $[u^{*},\infty)$. 
    It can be shown that for any $x>0$, 
    there exists a positive number $u(x)$ such that
    \begin{align}
    	 &\kappa(xu(x))=0 , \quad |u^{*}-u(x)|\leq{}\frac{2\pi}{x}, \quad{\rm and} \nonumber\\
	 \label{AuxEq5}
       &\int_{u(x)-\pi/x}^{u(x)}
        \kappa(ux)\phi(u)du
    	\leq q(x)\leq \int_{u(x)}^{u(x)+\pi/x}
        \kappa(ux)\phi(u)du,
    \end{align}
    (see e.g. Figure \ref{figureDefna} where the choice of $u(x)$ is illustrated 
    when $\kappa(u)=\cos(u)$). 
    Next, the upper and lower bounds on $q$ are such that:
    \begin{align*}
    	\int_{u(x)}^{u(x)+\pi/x}
        \kappa(ux)\phi(u)du\leq& \frac{\pi}{x}\phi(\bar{u}(x))\leq 
        \frac{\pi}{x} \phi(u^{*}),\\
        \int_{u(x)-\pi/x}^{u(x)}
        \kappa(ux)\phi(u)du\geq &
        -\frac{\pi}{x} \phi(\underline{u}(x))\geq
        -\frac{\pi}{x} \phi(u^{*}),
     \end{align*}
     where $\bar{u}(x)\in[u(x),u(x)+\pi/x]$ and $\underline{u}(x)\in[u(x)-\pi/x,u(x)]$.
     The inequality (\ref{EqEsyEst3}) is thus clear, while (\ref{Limit}) results from the fact that 
     both $\bar{u}(x)$ and $\underline{u}(x)$ converges to $u^{*}$ as $x\rightarrow\infty$.
\end{proof}
}

\begin{figure}[htb]
    {\par \centering
    \includegraphics[width=10cm,height=7cm]{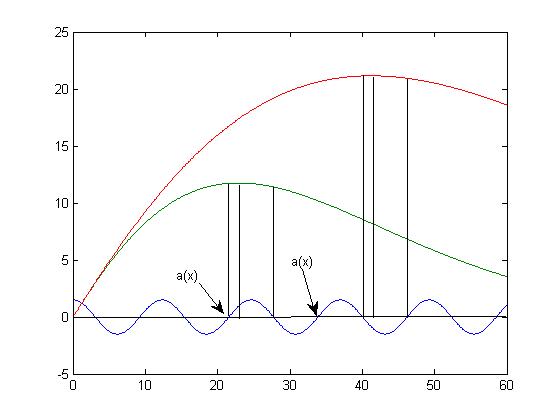}
    \par}
    \caption{\label{figureDefna}
    Definition of $a(x)$}
\end{figure}

\begin{prop}\label{PrpBndStbl}
    Let $X$ be a symmetric $\alpha$-stable L\'evy 
    process, and let $p_{t}$ be the density of the marginal $X_{t}$.
    {
    The following two statements hold:
    \begin{enumerate}
    \item[(a)] If $0<\alpha\leq 1$, then there exists an absolute  constant $c$ such that
    \[
        \sup_{x}|x|^{\alpha+1}p_{1}(x)\leq{}
        c.
    \]
    \item[(b)] If $1<\alpha\leq{}2$, then for any $\varepsilon>0$, there exists 
           a constant $0<c(\varepsilon)<\infty$ such that 
 \[
        \sup_{|x|>\varepsilon}|x|^{\alpha+1}p_{1}(x)\leq{}
        c(\varepsilon).
    \]   
    \end{enumerate}
    }
\end{prop}
\begin{proof}
    Without loss of generality suppose that $x>0$.
    By (\ref{SymmtInvForm}), the well-known representation of the characteristic function of $X_{t}$, and an 
    integration by parts,
    \begin{align}\label{AuxEq4}
        p_{1}(x)=
        \frac{1}{\pi}\displaystyle{\int_{0}^{\infty}}
        \cos(ux)e^{-u^{\alpha}} du
        =
        \frac{\alpha}{x}\int_{0}^{\infty}
        \sin(ux)u^{\alpha-1}e^{-u^{\alpha}} du.
    \end{align}
    If $0<\alpha\leq{}1$, then we can apply (\ref{EqEsyEst1&2}) with $\beta=1-\alpha$, and hence, 
    \begin{align*}
        \left|p_{1}(x)\right|&\leq \frac{\alpha}{x}\cdot \frac{c'}{x^{1-\beta}}= \frac{c}{x^{\alpha+1}},
    \end{align*}
    for a constant $c$.
    Now, let $1<\alpha\leq{}2$.  
    Applying another integration by parts in (\ref{AuxEq4}), 
    we have
    \begin{align}
        p_{1}(x)&=\frac{\alpha(\alpha-1)}{x^{2}}\int_{0}^{\infty}
        \cos(ux)u^{\alpha-2}e^{-u^{\alpha}} du \label{AuxEq5a}\\
        &\quad -\frac{\alpha^{2}}{x^{2}}\int_{0}^{\infty}
        \cos(ux)u^{2(\alpha-1)}e^{-u^{\alpha}} du.\label{AuxEq5}
    \end{align}
    The first term in the previous inequality can be bounded using  (\ref{EqEsyEst1&2}) 
    with $\beta=2-\alpha$:
    \begin{equation*}
    	\left|\frac{1}{x^{2}}\int_{0}^{\infty}
        \cos(ux)u^{\alpha-2}e^{-u^{\alpha}} du\right|
        \leq \frac{c}{x^{2}\cdot x^{1-\beta}}=\frac{c}{x^{\alpha+1}}.
    \end{equation*}
    The term in (\ref{AuxEq5}) can be bounded using  (\ref{EqEsyEst3}) since 
    $\phi(u)=u^{2(\alpha-1)}e^{-u^{\alpha}}$ is unimodal and thus, 
	\begin{align*}
		\left|\frac{\alpha^{2}}{x^{2}}\int_{0}^{\infty}
        \cos(ux)u^{2(\alpha-1)}e^{-u^{\alpha}} du\right|&\leq
        c\frac{1}{x^{3}}\leq{} c' \frac{1}{x^{1+\alpha}},
    \end{align*}
    for all $x>\varepsilon$, where $c,c'<\infty$ are constants depending only on $\varepsilon$. 
    Plugging in the above bounds in (\ref{AuxEq5}), we obtain the second statement in the proposition.    
\end{proof}

\begin{remk}
    In view of the above proposition, we obtain the
    following bound for the transition density $p_{t}$ of
    a symmetric $\alpha-$stable L\'evy process:
    \[
        p_{t}(x)\leq{} c \frac{t}{x^{\alpha+1}},
    \]
    valid for all $t>0$ and $|x|>\varepsilon$,
    and where $c$ is a constant depending only on $\varepsilon$.
\end{remk}

We can now generalize the ideas of Proposition \ref{PrpBndStbl}
to dealt with the derivatives of the transition density.
\begin{thrm}\label{PrpDrvBndStbl}
    Under the conditions of Proposition \ref{PrpBndStbl},
    for any $\varepsilon>0$, there exists a constant $c_{n}(\varepsilon)$
    such that
    \begin{equation}\label{EqDrvBndStbl0}
        \sup_{|x|>\varepsilon}
        |x|^{\alpha+1+n}\left|p_{1}^{(n)}(x)\right|
        \leq{}
        c_{n}(\varepsilon).
    \end{equation}
\end{thrm}
\begin{proof}
    We prove the following more general bound:
    \begin{align}\label{EqDrvBndStbl}
        \sup_{|x|>\varepsilon}
        |x|^{\alpha+1+n}&\left|\displaystyle{\int_{0}^{\infty}}
        \kappa(ux)u^{n}e^{-u^{\alpha}} du\right|
        \leq{}
        d_{n}(\varepsilon)<\infty,
    \end{align}
    where $\kappa$ can be
    either $\cos$ or $\sin$.
    Without loss of generality, let us assume that $x>0$.
    Our proof is then performed by induction on $n$.
    Proposition \ref{PrpBndStbl} yields
    (\ref{EqDrvBndStbl}) for $n=0$ and $\kappa(x)=\cos(x)$.
    The case $\kappa(x)=\sin(x)$ can be dealt with in an analogous way; namely, we first integrate 
    by parts, once when $\alpha\leq{}1$, or twice when $1<\alpha\leq{}2$, and secondly, 
    we use (\ref{EqEsyEst1&2})  if $\alpha\leq{}1$, or 
   (\ref{EqEsyEst3}) if $1<\alpha\leq{}2$.

    Now, assume that (\ref{EqDrvBndStbl})
    holds for $n=0,\dots,m-1$. We want to prove the case
    $n=m>1$.  Set
     \begin{align*}
        q_{m}(x):=
        \displaystyle{\int_{0}^{\infty}}
        \kappa(ux)u^{m}e^{-u^{\alpha}} du.
    \end{align*}
    Applying consecutive integrations by parts,
    one can find constants $b_{j}$
    (depending only on $\alpha$ and $m$) such that
     \begin{equation}\label{AuxEq1}
        q_{m}(x)=
        -\frac{1}{x}\displaystyle{\int_{0}^{\infty}}
        \hat{\kappa}(ux)u^{m-1}e^{-u^{\alpha}} du
        +\frac{1}{x^{m}}
        \sum_{j=1}^{m} b_{j}
        \int_{0}^{\infty}
        \bar{\kappa}(ux)u^{i\alpha}e^{-u^{\alpha}} du,
    \end{equation}
    where $\hat{\kappa},\bar{\kappa}$ are either $\cos$ or $\sin$.
    By the induction hypothesis, the first term in (\ref{AuxEq1})
    is such that
    \begin{equation}\label{AuxEq2}
        \sup_{|x|>\varepsilon}
        \left|\frac{1}{x}\displaystyle{\int_{0}^{\infty}}
        \hat{\kappa}(ux)u^{m-1}e^{-u^{\alpha}} du
        \right|
        |x|^{\alpha+1+m}
        \leq{} d_{m-1}(\varepsilon),
    \end{equation}
    as we wanted to show.

 	\vspace{.2 cm}
    \noindent Now, for the second term, let us consider first $\alpha<1$. Let $k\geq{}1$ be such that
    \[
        \frac{k-1}{k-1+m}< \alpha \leq{}\frac{k}{k+m}.
    \]
    Also, for each $1\leq{}j\leq{}m+k$, let
    $1\leq r_{j}\leq j$ be such that
    \[
        \frac{r_{j}-1}{j}< \alpha \leq{}\frac{r_{j}}{j}.
    \]
    Setting $S(x):=\sum_{j=1}^{m} b_{j}
        \int_{0}^{\infty}
        \kappa(ux)u^{j\alpha}e^{-u^{\alpha}} du$, and 
    applying successive integrations by parts to each of
    the terms of $S(x)$, it follows that
     \begin{equation}\label{AuxEq10}
        S(x)=\sum_{j=1}^{m+k} {a_{j}}{x^{r_{j}}}
        \int_{0}^{\infty}
        \kappa_{j}(ux)u^{j\alpha-r_{j}}e^{-u^{\alpha}} du
    \end{equation}
    for some constants $a_{j}$, and where
    $\kappa_{j}$ is either $\cos$ or $\sin$.
    By the way $r_{j}$ is chosen, the inequality (\ref{EqEsyEst1&2}) can be applied 
    to estimate the absolute value of each term in (\ref{AuxEq10}). Then,
    \begin{align}
        |S(x)|=
        \sum_{j=1}^{m+k} \frac{a_{j}}{x^{r_{j}}}
        \left|\int_{0}^{\infty}
        \kappa_{j}(ux)u^{j\alpha-r_{j}}e^{-
        u^{\alpha}} du\right|\leq
        \sum_{j=1}^{m+k} \frac{\hat{a}_{j}}{x^{j\alpha +1}},
        \label{AuxEq3}
    \end{align}
    for some $\hat{a}_{j}\geq{}0$.
    Combining (\ref{AuxEq1})-(\ref{AuxEq3}),
    there exists a constant $c_{m}(\varepsilon)$ such that
    \begin{equation}\label{AuxEq2}
        \sup_{|x|>\varepsilon}
        \left|p_{1}^{(m)}(x)\right|
        |x|^{\alpha+1+m}
        \leq{} c_{m}(\varepsilon).
    \end{equation}

 	\vspace{.2 cm}
    \noindent  Next, we consider the case of $1<\alpha\leq 2$. 
     Note that, for some constants $a_{0},a_{1},a_{2}$ depending only on $\alpha$ and $j$, 
     the term 
     $C_{j}(x):=
     	 \int_{0}^{\infty}
        \kappa(ux)u^{j\alpha}e^{-u^{\alpha}} du$ can be broken into three pieces:
     \begin{align}
     	 C_{j}(x)&=\frac{a_{0}}{x^{2}}\int_{0}^{\infty}
        \kappa(ux)u^{j\alpha-2}e^{-u^{\alpha}} du\label{AuxEq6}\\
        &\quad+     	 
        \frac{a_{1}}{x^{2}}\int_{0}^{\infty}
        \kappa(ux)u^{(j+1)\alpha-2}e^{-u^{\alpha}} du+
             	 \frac{a_{2}}{x^{2}}\int_{0}^{\infty}
        \kappa(ux)u^{(j+2)\alpha-2}e^{-u^{\alpha}} du.\nonumber
      \end{align}
      If $j=1$, then the first term in (\ref{AuxEq6}) can be bounded using  (\ref{EqEsyEst1&2})
      with $\beta=2-\alpha$:
        \[
        	 \left|\frac{a_{0}}{x^{2}}\int_{0}^{\infty}
        	\kappa(ux)u^{\alpha-2}e^{-u^{\alpha}} du\right|
        	\leq{}\frac{c}{x^{2} x^{1-\beta}}=\frac{c}{x^{\alpha+1}},
        \]
        for some $c<\infty$. For the other two terms of the case $j=1$ or any other $2\leq{}j\leq{}m$,
        we can apply (\ref{EqEsyEst3}) since then the function multiplying $\kappa$ is unimodal. Then, 
        for any $\varepsilon>0$, we can bound $S(x):=\sum_{j=1}^{m} b_{j}
        \int_{0}^{\infty}
        \kappa(ux)u^{j\alpha}e^{-u^{\alpha}} du$ in the following way:
        \[
        	\sup_{x>\varepsilon}|S(x)|\leq{} 
		\frac{c}{x^{\alpha+1}}+\frac{c'}{x^{3}}\leq{} \frac{{c''}}{x^{\alpha+1}},
	\]
	for a constant ${c''}$ depending only on $\varepsilon$.
        
        Finally, let us verify the case $\alpha=1$.
    Without loss of generality, assume that $\kappa(x)=\cos(x)$.
    After two integrations by parts,
    we have that
     \begin{align*}
        \displaystyle{\int_{0}^{\infty}}
        \cos(ux)u^{m}e^{-u} du&=
        -\frac{m}{x}\displaystyle{\int_{0}^{\infty}}
        \sin(ux)u^{m-1}e^{-u} du+
        \frac{1}{x}\displaystyle{\int_{0}^{\infty}}
        \sin(ux)u^{m}e^{-u} du\\
        &=
        -\frac{m}{x}\displaystyle{\int_{0}^{\infty}}
        \sin(ux)u^{m-1}e^{-u} du-\frac{1}{x^{2}}\displaystyle{\int_{0}^{\infty}}
        \cos(ux)u^{m}e^{-u} du\\
        \quad&+
        \frac{m}{x^{2}}
        \displaystyle{\int_{0}^{\infty}}
        \cos(ux)u^{m-1}e^{-u} du.
    \end{align*}
    We can then write the above equality in the following manner:
    \begin{align*}
        \left(1+\frac{1}{x^{2}}\right)
        \displaystyle{\int_{0}^{\infty}}
        \cos(ux)u^{m}e^{-u} du
        &=-\frac{m}{x}\displaystyle{\int_{0}^{\infty}}
        \sin(ux)u^{m-1}e^{-u} du\\
        \quad&
        +\frac{m}{x^{2}}\displaystyle{\int_{0}^{\infty}}
        \cos(ux)u^{m-1}e^{-u} du.
    \end{align*}
    The result follows by applying our induction hypothesis to bound each
    of the two terms in the right-hand side of the last equality.
\end{proof}

\begin{cllry}
    With the notation of Proposition \ref{PrpDrvBndStbl},
    for any $0<\alpha\leq{}2$, $\varepsilon>0$, and
    $n\geq{}0$, there exist a constant
    $c_{n,\varepsilon}$ such that
    \begin{equation}\label{EqDrvBndStblt}
        \sup_{|x|>\varepsilon}
        \left|p_{t}^{(n)}(x)\right|
        \leq{}
        c_{n,\varepsilon} t,
    \end{equation}
    for any $0<t\leq{}1$.
\end{cllry}

{
\section{General L\'evy processes}\label{SectGnrlSymmtLevy}

In this part, we examine the validity of the assumption
(\ref{StndgCond}) for general L\'evy processes, whose L\'evy density $s$ is stable like around the origin. 

The main tool will be a recursive relations between the derivatives of a density $p$. Consider a distribution $\mu$ such that its characteristic function $\psi(u):=\hat{\mu}(u)$ is $C^{\infty}$ with also 
\begin{equation}\label{CndChrFnct}
	\int_{-\infty}^{\infty}|u|^{m}\left|\psi^{(r)}(u)\right|du<\infty,
\end{equation}
for all $r\geq 0$ and $m\geq{}0$. Recall that in that case $\mu$ admits a $C^{\infty}$-density $p$ and moreover,
\begin{equation}\label{InvFormb}
    p^{(m)}(x)=\frac{(-i)^{m}}{2\pi}
    \int e^{-iux}u^{m}\psi(u)du.
\end{equation}
By applying two consecutive integration by parts, we can derive the following formulas
\begin{align*}
	p^{(m)}(x)&=- \frac{m}{x} \, {p}^{(m-1)}- \frac{(-i)^{m-1}}{2\pi x} 
	\int e^{-iux} u^{m}\frac{d{\psi(u)}}{d u} du,\\
	{p}^{(m)}(x)&= -2\frac{m}{x} \, {p}^{(m-1)}-\frac{m(m-1)}{x^{2}} \, {p}^{(m-2)}
	+ \frac{(-i)^{m-2}}{2\pi x^{2}} 
	\int e^{-iux} u^{m}\frac{d^{2} {\psi(u)}}{d u^{2}}du,
\end{align*}
where we are assuming that $m\geq{}2$. However, even if $m<2$, we can deduce a recursive formula for ${p}^{(m)}$ in terms of all its lower order derivatives and the integral  of the function 
\[
	e^{-iux}u^{m}\frac{d^{r} \psi(u)}{d u^{r}}.
\]
Indeed, we have:
\begin{thrm}\label{RcrsFormThrm}
Let $r\geq 0$ and $m\geq{}0$.
Then, for all $x$, ${p}^{(m)}(x)$ can be written as
\begin{align*}
	&\sum_{j=1}^{r \wedge m} c_{r,j}^{m}\, \prod_{i=0}^{j-1}(m-i) \frac{1}{x^{j}} 
	\,{p}^{(m-j)}(x)+ (-1)^{r}
	 \frac{(-i)^{m-r}}{2\pi x^{r}} 
	\int e^{-iux} u^{m}\frac{d^{r}\psi(u)}{d u^{r}}du,
\end{align*}
where $c_{i,j}^{m}$ are given by the following recursive formulas:
\begin{align}\label{MaimRcrsvForm}
	c_{r,0}^{m}=-1,\quad  c_{r,j}^{m}=0,\; \; (j>r),\quad
	c_{r+1,j}^{m}&=c_{r,j}^{m}+c_{r,j-1}^{m-1}.
\end{align}
\end{thrm}
\begin{proof}
	We prove the formula by induction in $m$. Consider the case $m=0$. We want to prove that 
	\begin{align*}
	{p}(x)&=(-1)^{r}
	 \frac{(-i)^{-r}}{2\pi x^{r}} 
	\int e^{-iux} \frac{d^{r}}{d u^{r}} {\psi(u)}du,
	\end{align*}
	for any $r\geq{}0$. This can be done by induction on $r$ and integration by parts.
	Suppose that the formula is valid for $m=k$ and all $r\geq{}0$. We want to show the formula for 
	$m=k+1$ and all $r\geq0$. 
	Now, we use induction on $r$. The case $r=0$ is just (\ref{InvFormb}) with $m=k+1$. 
	Suppose the result holds for $r=\ell $ and $m=k+1$:
\begin{align}\label{RecrvForAux}
	{p}^{(k+1)}(x)&=\sum_{j=1}^{\ell\wedge (k+1)} c_{\ell,j}^{k+1}\, \prod_{i=0}^{j-1}(k+1-i) \frac{1}{x^{j}} 
	\,{p}^{(k+1-j)}(x)\\
		&\quad+ (-1)^{\ell}
		 \frac{(-i)^{k+1-\ell}}{2\pi x^{\ell}} 
		\int e^{-iux} u^{k+1}\frac{d^{\ell}}{d u^{\ell}} {\psi(u)}du.\nonumber
\end{align}
Next, with an integration by parts in the last term,
\begin{align}\label{RecrvForAux1}
	{p}^{(k+1)}(x)&=\sum_{j=1}^{\ell\wedge (k+1)} c_{\ell,j}^{k+1}\, \prod_{i=0}^{j-1}(k+1-i) \frac{1}{x^{j}} 
	\,{p}^{(k+1-j)}(x)\\
	&\quad+ (-1)^{\ell+1}
	 \frac{(-i)^{k+1-\ell-1}}{2\pi x^{\ell+1}} 
	\int e^{-iux} u^{k+1}\frac{d^{\ell+1}}{d u^{\ell+1}} {\psi(u)}du\nonumber\\
	&\quad +(-1)^{\ell+1}
	 \frac{(-i)^{k+1-\ell-1}}{2\pi x^{\ell+1}}\cdot (k+1)
	\int e^{-iux} u^{k}\frac{d^{\ell}}{d u^{\ell}} {\psi(u)}du.\nonumber
\end{align}
Then, writing (\ref{MaimRcrsvForm}) for $m=k$ and $r=\ell$ and solving for the last term gives
\begin{align*}
	& (-1)^{\ell+1}
		 \frac{(-i)^{k-\ell}}{2\pi x^{\ell}} 
		\int e^{-iux} u^{k}\frac{d^{\ell}}{d u^{\ell}} {\psi(u)}du=\\
	&-{p}^{(k)}(x)+\sum_{j=1}^{\ell\wedge k} c_{\ell,j}^{k}\, \prod_{i=0}^{j-1}(k-i) \frac{1}{x^{j}} 
	\,{p}^{(k-j)}(x)
\end{align*}
Plugging in (\ref{RecrvForAux1}), we get (\ref{MaimRcrsvForm})  with $r=\ell+1$ and $m=k+1$ 
provided that we define the  coefficients $c^{k+1}_{\ell+1,j}$ as follows:
\begin{align*}
	c_{\ell+1,1}^{k+1}:=c_{\ell,1}^{k+1}-1,\quad c_{\ell+1,j}^{k+1}:=c_{\ell,j}^{k+1}+c_{\ell,j-1}^{k}.
\end{align*}
This proves the case of $r=\ell+1$ and so, the result holds for all $r$ and all $m$.
\end{proof}

The following corollary give further information when working with the transition distributions of a L\'evy process.
\begin{cllry}\label{RprstnDnstyLvy}
Let $(X_{t})_{t\geq{}0}$ be a L\'evy process such that $\mu$, the distribution of $X_{1}$, satisfies (\ref{CndChrFnct}). 
Let $\gamma$ be such that
\(
	\psi_{t}(x):=e^{t \gamma(u)},
\)
where $\psi_{t}$ is  the characteristic function of $X_{t}$. Then, the density $p_{t}$ of $X_{t}$ admits the representation:
\begin{align}\label{RcrsFrml}
	{p}_{t}^{(m)}(x)&=\sum_{j=1}^{r \wedge m} c_{r,j}^{m}\, \prod_{i=0}^{j-1}(m-i) \frac{1}{x^{j}} 
	\,{p}_{t}^{(m-j)}(x)+ (-1)^{r}
	 \frac{(-i)^{m-r}}{2\pi x^{r}} \,\calI^{m}_{r}(t,x),
\end{align}
where 
\[
	\calI^{m}_{r}(t,x):=\sum_{(i_{1},i_{2},j_{1},j_{2})} d_{i_{1},i_{2}}^{j_{1},j_{2}}\cdot
	t^{j_{1}+j_{2}}\int e^{-iux}
	\left(\gamma^{(i_{1})}(u)\right)^{j_{1}}\left(\gamma^{(i_{2})}(u)\right)^{j_{2}}
	e^{t\gamma(u)} 	du,
\]
for some constants $ d_{i_{1},i_{2}}^{j_{1},j_{2}}$. The above summation is over all non-negative integers $i_{1},i_{2},j_{1},j_{2}$ such that 
$0<i_{2}\leq i_{1}$ and
\(
	i_{1}j_{1}+i_{2}j_{2}=r.
\)
\end{cllry}

As an application let us consider a L\'evy process as in Corollary \ref{RprstnDnstyLvy} such that
for each $i\geq 1$, there exists $c_{i}<\infty$ and $u_{0,i}>0$ such that 
\begin{equation}\label{CndDrvCmmlnt}
	|\gamma^{(i)}(u)|\leq c_{i} |u|^{\alpha-i},
\end{equation} 
for all $|u|>u_{0,i}$. Also, assume that there exists $u_{0}>0$ and $c_{0}<\infty$ such that 
\begin{equation}\label{BndChrFnctn}
	|\psi_{1}(u)|\leq e^{-c_{0} |u|^{\alpha}},
\end{equation}
for all $u>u_{0}$. Remember that (\ref{CndtnDnsty}) implies the above condition (cf. Sato \cite[Proposition 28.3]{Sato}). Then, we have the following result:
\begin{prop}
Let (\ref{CndDrvCmmlnt}) and (\ref{BndChrFnctn}) be true for $0<\alpha\leq{}2$. Then,
for any $m\geq{}0$, any $\varepsilon>0$, and any $t_{0}>0$,
    \[
        \sup_{0<t\leq{}t_{0}}
        \sup_{|x|>\varepsilon}|p_{t}^{(m)}(x)|<\infty.
    \]
\end{prop}  
\begin{proof}
The proof is by induction on $m\geq 0$. The recursive formula (\ref{RcrsFrml}) with $r=1$ and $m=0$ leads to
\(
	|p_{t}(x)|\leq{}  t \left|\int e^{-iux}
	\gamma'(u)
	e^{t\gamma(u)} 	du\right|/x.
\)
Note that we can assume that there exist constants $u_{0}>0$, $b_{0}$,  and $b_{1}$ such that 
\begin{align*}
	\sup_{|w|\leq{}u_{0}}|\gamma'(w)|
	|e^{t\gamma(w)}|\leq b_{0}, \quad 
	|\gamma'(u)\cdot e^{t\gamma(u)}|\leq b_{1} |u|^{\alpha-1}e^{-c_{0}t|u|^{\alpha}},
\end{align*}
for all $|u|>u_{0}$ and $0<t\leq t_{0}$. Then, for all $t\leq{}t_{0}$,
\begin{align*}
	|p_{t}(x)|&\leq b_{0} u_{0} \frac{t}{x} +
	b_{1}\int_{0}^{\infty} v^{\alpha-r}e^{-c_{0}v^{\alpha}}dv.
\end{align*}
Next, let the statement of the proposition hold true for $m=0,\dots,k$, and let us show it for $m=k+1$. 
In view of (\ref{RcrsFrml}), it suffices to show that 
  \[
        \sup_{0<t\leq{}t_{0}}
        \sup_{|x|>\varepsilon}|\calI_{r}^{m}(t,x)|<\infty,
    \]
for some $r\geq{}0$. 
Moreover, it suffices to show that
\[
	\sup_{0<t\leq{}t_{0}}
        \sup_{|x|>\varepsilon}
	t^{j_{1}+j_{2}}\left|\int e^{-iux} u^{m}
	(\gamma^{(i_{1})}(u))^{j_{1}}(\gamma^{(i_{2})}(u))^{j_{2}}
	e^{t\gamma(u)} 	du\right|<\infty,
\]
for any $i_{1}\geq i_{2}>0$ and $j_{1},j_{2}\geq 0$  such that 
$i_{1}j_{1}+i_{2}j_{2}=r$. As before,  we can assume that there exist constants $u_{0}>0$, $b_{0}$,  and $b_{1}$ such that 
\begin{align*}
	\sup_{|u|\leq{}u_{0}}|\gamma^{(i_{1})}(u)|^{j_{1}}|\gamma^{(i_{2})}(u)|^{j_{2}}
	|e^{t\gamma(u)}|&\leq b_{0}\\
	|\gamma^{(i_{1})}(u)|^{j_{1}}|\gamma^{(i_{2})}(u)|^{j_{2}}
	|e^{t\gamma(u)}|&\leq b_{1} |u|^{(j_{1}+j_{2})\alpha-r}e^{-c_{0}t|u|^{\alpha}},
\end{align*}
for all $|u|>u_{0}$. We need to show that there exists an $r$ such that the supremum on $0<t<t_{0}$ of 
\begin{align*}
	t^{j_{1}+j_{2}}\int_{0}^{\infty} u^{(j_{1}+j_{2})\alpha+m-r}e^{-c_{0}t u^{\alpha}}du=
	t^{\frac{1}{\alpha}\left(r-m-1\right)}\int_{0}^{\infty} v^{(j_{1}+j_{2})\alpha+m-r}e^{-c_{0}v^{\alpha}}dv.
\end{align*}
is finite.  The supremum above will be finite if $r=m+1$.
\end{proof}

\begin{exmpl}
	Consider the CGMY L\'evy model introduced in \cite{Madan} and of  {great popularity in the area of mathematical finance}. This process is a tempered stable one in the sense of Rosi\'nski \cite{Rosinski:2007}. Its characteristic function is given by 
	\[
		\psi_{t}(u)=\exp\left\{t C\Gamma(-\alpha)\left((M-iu)^{\alpha}-M^{\alpha} +
		(G+iu)^{\alpha}-G^{\alpha}\right)\right\}
	\]
	(see Theorem 1 in \cite{Madan}). Then, 
	\[
		\gamma(u):=C\Gamma(-\alpha)\left((M-iu)^{\alpha}-M^{\alpha} +
		(G+iu)^{\alpha}-G^{\alpha}\right).
	\]
	We can then verify that $\gamma$ satisfies  (\ref{CndDrvCmmlnt}) and (\ref{BndChrFnctn}).
\end{exmpl}}

{The next result generalizes the conclusions in the above example to more general tempered stable processes. 
For simplicity, we take symmetric processes, even though the proof can be extended to the general case.
\begin{prop}\label{PropGnrlTSP}
	Let $X$ be a L\'evy process with L\'evy triplet $(0,0,\nu)$. Assume that $\nu$ is of the form
	\( 
		\nu(ds)=|s|^{-\alpha-1}q(|s|) ds,
	\) 
	where $0<\alpha<{}2$ and $q$ is a completely monotone function on $\bbr_{+}$ such that 
	\begin{equation}\label{MntCnd}
		\int_{1}^{\infty} s^{j-\alpha-1}q(s)ds<\infty,
	\end{equation}
	for all $j\geq{}1$.
	Assume also that the measure $F$ for which 
	\(
		q(s)=\int_{0}^{\infty} e^{-\lambda s} F(d\lambda)
	\)
	is such that
	\begin{equation}\label{CndF}
		\int_{0}^{\infty} \lambda^{j} F(d\lambda)<\infty,
	\end{equation}
	for all $j\geq{}0$. Then, the function $\gamma$ associated with the characteristic function of $X$ via $\psi_{t}(x):=e^{t \gamma(u)}$ satisfies the conditions (\ref{CndDrvCmmlnt}) and (\ref{BndChrFnctn}). 
\end{prop}
\begin{proof}
	Clearly, 
    \begin{equation}
        \liminf_{\varepsilon\rightarrow{}0}
        \frac{\int_{0}^{\varepsilon} s^{1-\alpha} q(s) ds}
        {\varepsilon^{2-\alpha}}
        >0,
    \end{equation}
    and thus, condition (\ref{BndChrFnctn}) will follow.
    Now, we claim that there exists a constant $C$ such that
    \begin{equation}\label{AuxEst10}
    	\left|\int_{0}^{\infty} \sin(us) s^{-\alpha} e^{-\lambda s}d s \right|\leq{}  C u^{\alpha-1},
    \end{equation} 
    for all $\lambda, u>0$ and $0<\alpha<2$. Indeed, if $0<\alpha\leq 1$, (\ref{AuxEst10}) results from (\ref{EqEsyEst1&2}). If $1\leq{}\alpha<2$, then changing variables and using $\sin v\leq v$, 
    \begin{align*}
    	\left|\int_{0}^{\infty} \sin(us) s^{-\alpha} e^{-\lambda s}d s \right|&\leq{}  
	u^{\alpha-1}\left|\int_{0}^{\infty} \sin(v) v^{-\alpha} e^{-\lambda v/u}d v \right|\\
	&\leq 
	u^{\alpha-1}\int_{0}^{\pi}  v^{1-\alpha} dv +
	u^{\alpha-1}\int_{\pi}^{\infty}  v^{-\alpha} d v\leq C u^{\alpha-1},
    \end{align*}
    for a constant $C$ independent of $u$ and $\lambda$. Moreover, it can be proved that 
    there exists a constant $C_{j}$ such that 
    \begin{equation}\label{AuxEst1b}
    	\left|\int_{0}^{\infty} \kappa(us) s^{j-\alpha} e^{-\lambda s}d s \right|\leq{}  C_{j} (1+\lambda)^{j}u^{\alpha-1-j},
    \end{equation} 
    for $j\geq{}1$, $\lambda, u>0$, and $0<\alpha<2$, and where $\kappa$ can be either $\cos$ or $\sin$.
    Indeed, the case $j=1$ can be proved as follows. If $0<\alpha\leq{}1$, then we apply two times integration by parts (similar to the case $\alpha=1$ in the proof of Theorem 4.4). Then, we can apply part (i) of Lemma \ref{lmEsyEst}. If $1<\alpha\leq{}2$, then one can apply directly part (i) of Lemma \ref{lmEsyEst}. 
    The case $j\geq{}1$ can be proved using induction on $j$ with the help of two integration by parts. 
    From the previous estimates, we have that, for $j\geq{}0$,
        \begin{equation}\label{AuxEst1}
    	\left|\int_{0}^{\infty} \kappa_{j}(us) s^{j-\alpha} e^{-\lambda s}d s \right|\leq{}  C_{j} (1+\lambda)^{j}u^{\alpha-1-j},
    \end{equation} 
    where $C_{j}$ is a constant independent of $\lambda$ and $u$ and  $\kappa_{j}(u)=\cos(u)$ if $j$ is odd, and 
  $\kappa_{j}(u)=\sin(u)$ if $j$ is even. 
   Next, from the conditions on $X$, the function $\gamma$ is given by
    \(
    	\gamma(u)=2\int_{0}^{\infty}(1-\cos u s) s^{-\alpha-1} q(s) ds.
    \)
    Condition (\ref{MntCnd}) implies that 
    \begin{align*}
    	\left|\gamma^{(j)}(u)\right|&=2\left|\int_{0}^{\infty}\kappa_{j-1}(us) s^{j-\alpha-1} q(s) ds\right|,
	\quad j\geq{}1, 
    \end{align*}
    where $\kappa_{j}$ is as above.
    In that case, using (\ref{CndF}), applying Fubini's Theorem, and (\ref{AuxEst1}), we have
    \begin{align*}	
    	\left|\gamma^{(j)}(u)\right|&\leq 2\int_{0}^{\infty}\left|\int_{0}^{\infty}\kappa_{j-1}(us) s^{j-\alpha-1} e^{-\lambda s} ds \right|F(d\lambda)\\
	&\leq 2C_{j}u^{\alpha-j}\ \int_{0}^{\infty}(1+\lambda)^{j-1}F(d\lambda)
	\leq c_{j}u^{\alpha-j},
    \end{align*}
    for a constant $c_{j}$ independent of $u$.
 \end{proof}

\begin{remk}
	 Rosi\'nski \cite{Rosinski:2007} (see Proposition 2.7)  gives conditions for (\ref{MntCnd}) to hold. 
	 In terms of the notation of Proposition \ref{PropGnrlTSP},  (\ref{MntCnd})  holds with $j>1$ if and only if 
	 \(
	 	\int_{0}^{1} \lambda^{-j} F(d\lambda)<\infty,
	\)
	which is also necessary and sufficient for $j=1$ provided that $\alpha<1$. 
	If $\alpha>1$, then  (\ref{MntCnd})  always hold for $j=1$, while when $\alpha=1$, 
	(\ref{MntCnd})  hold with $j=1$ if and only if 
	\(
	 	\int_{0}^{1} \lambda^{-1} \log(\lambda^{-1})F(d\lambda)<\infty.
	\)
\end{remk}	
}

\appendix
\section{\textbf{Verification of the claim in Remark \ref{Simplify} (iii).}}\label{TechnicalProofs}
{\DRed Note that the expression for $d_{2}$ in Remark  \ref{Simplify} (ii) can be modified so that one can replace $\bar{s}_{\varepsilon}(x)$ by $s(x){\bf 1}_{\{|x|\geq \varepsilon\}}$. We will get:}
\begin{align}
    d_{2}&=
    -\sigma^{2}s'(y)+2 b s(y)-2
        \int_{|w|\leq{}\varepsilon} \int_{0}^{1}s'(y-\beta w)
        (1-\beta)d\beta  w^{2} s(w)d w
    \label{Line2}\\
    &\quad+\int_{|x|\geq{}\varepsilon}\int_{|u|\geq{}\varepsilon}
    {\bf 1}_{\{x+u\geq{}y\}} s(x)s(u)dxdu
    \label{Line3}\\
    &\quad
    -2\int_{|x|\geq{}\varepsilon}s(x)dx \int_{y}^{\infty} s(x)dx
    -2\int_{\varepsilon\leq{}|x|\leq{}1}x s(x)dx s(y).
    \label{Line4}
\end{align}
The last term in (\ref{Line2}) converges to $0$ as $\varepsilon\rightarrow{}0$.
The term in (\ref{Line3}) can be written as follows (omitting 
the integrand $s(x)s(u)$ and using  symmetry of this about  $x=u$):
\begin{align*}
    &\underbrace{2\int_{-\infty}^{-\varepsilon} dx \int_{y-x}^{\infty}
    du}_{A_{1}}+\underbrace{2\int_{y}^{\infty} dx\int_{\varepsilon}^{\infty} du}_{A_{2}}
    -\underbrace{\int_{y}^{\infty} dx \int_{y}^{\infty} du}_{A_{3}}\\
    &+
    \underbrace{\int_{y/2}^{y} dx\int_{y/2}^{y} du}_{A_{4}}
    +\underbrace{2\int_{\varepsilon}^{y/2} dx\int_{y-x}^{y} du}_{A_{5}}.
\end{align*}
Similarly, we can decompose the terms in line (\ref{Line4}) as
\begin{align*}
    &\underbrace{-2\int_{-1}^{-\varepsilon}s(x)dx \int_{y}^{\infty} s(u) du}_{B_{1}}
    -\underbrace{2\int_{-\infty}^{-1}s(x)dx \int_{y}^{\infty} s(u)du}_{B_{2}}\\
    &\underbrace{-2\int_{y}^{\infty}s(x)dx \int_{\varepsilon}^{\infty} s(u) du}_{B_{3}}
    \underbrace{-2s(y)\int_{-1}^{-\varepsilon} x s(x) dx}_{B_{4}}
    \underbrace{-2s(y)\int_{\varepsilon}^{1} x s(x) dx}_{B_{5}}
\end{align*}
Now, $A_{2}+B_{3}=0$, $A_{1} +B_{1}+B_{2}+B_{4}$ becomes
\begin{align*}
    2\int_{-1}^{-\varepsilon}
    \int_{y-x}^{y} \left\{s(u)-s(y)\right\} du s(x) dx
    +2\int_{-\infty}^{-1}
    \int_{y-x}^{y} s(u) du s(x) dx,
\end{align*}
and $A_{5}+B_{5}$ becomes 
\begin{align*}
 2\int_{\varepsilon}^{y/2}\int_{y-x}^{y} \left\{s(u)-s(y)\right\}
    du s(x) dx-2
    s(y)\int_{y/2}^{1} x s(x) d x.
\end{align*}
Then, after changing variables, $d_{2}$ becomes:
\begin{align*}
    &
    -\sigma^{2}\,s'(y)+  2 s(y)b+
    2\int_{-\varepsilon}^{\varepsilon}
        \int_{0}^{x}\left\{s(y-u)-s(y)\right\}du s(x)dx
-\nu([y,\infty))^{2}\\
& +\int_{y/2}^{y} s(x)dx \int_{y/2}^{y} s(u) du
    +2\int_{-\infty}^{-1}
    \int_{y-x}^{y} s(u) du s(x) dx
    -2s(y)\int_{y/2}^{1} x s(x) d x\\
    &
    + 2   \int_{-1}^{-\varepsilon}
    \int_{0}^{x} \left\{s(y-u)-s(y)\right\} du s(x) dx+
    2\int_{\varepsilon}^{y/2}\int_{0}^{x} \left\{s(y-u)-s(y)\right\}
    du s(x) dx.
\end{align*}
Now, taking $\varepsilon\rightarrow{}0$ in the above and further manipulation, gives the expression  in Remark \ref{Simplify} (iii).


\end{document}